\documentclass[leqno,11pt]{amsart}
\usepackage{amsmath,amsthm,amssymb,enumerate}
\usepackage{epsfig}
\usepackage{amssymb}
\usepackage{amsmath,xspace}
\usepackage{amssymb}
\usepackage{mathtools}
\usepackage{amsmath,amsthm}
\usepackage[latin1]{inputenc}
\usepackage[T1]{fontenc}
\usepackage{path}
\usepackage{ae,aecompl}
\usepackage{amsfonts}
\usepackage{amsxtra}
\usepackage{bbm,euscript,mathrsfs}
\usepackage{color}
\usepackage{todonotes}

\allowdisplaybreaks

\usepackage{amsfonts}
\usepackage{amsxtra}

\relax

\usepackage{enumitem}
\setenumerate{label={\rm (\alph{*})}}

\makeatletter
\long\def\unmarkedfootnote#1{{\long\def\@makefntext##1{##1}\footnotetext{#1}}}
\makeatother

\newcommand{\dd}{\mathrm{d}}

\newcommand{\dz}{\,\mathrm{d}z}
\newcommand{\dy}{\,\mathrm{d}y}
\newcommand{\dx}{\,\mathrm{d}x}

\newcommand{\dt}{\,\mathrm{d}t}
\newcommand{\dr}{\,\mathrm{d}r}

\newcommand{\R}{\mathbb R}
\newcommand{\N}{\mathbb N}

\newcommand{\rn}{\mathbb R^n}

\DeclareMathOperator{\tr}{tr}

\newtheorem{defs}{Definition}[section]
\newtheorem{theorem}[defs]{Theorem}

\newtheorem{example}[defs]{Example}

\newtheorem{lemma}[defs]{Lemma}

\newtheorem{remark}[defs]{Remark}

\title
[Inclusion relations among
fractional Orlicz-Sobolev spaces]{
Inclusion relations among
fractional Orlicz-Sobolev spaces and a Littlewood-Paley characterization
}

\author {Dominic Breit \& Andrea Cianchi}

\address{Dominic Breit, 
  TU Clausthal, Institute of Mathematics\\
Erzstra\ss e 1, Clausthal-Zellerfeld, Germany} \email{dominic.breit@tu-clausthal.de}

\address{Andrea Cianchi, Dipartimento di Matematica e Informatica \lq\lq U. Dini"\\
Universit\`a di Firenze,
Viale Morgagni 67/a,
50134 Firenze,
Italy} \email{andrea.cianchi@unifi.it}

\urladdr{}

\numberwithin{equation}{section}

\begin{document}
\begin{abstract} Embeddings among fractional Orlicz-Sobolev spaces with different smoothness are characterized. The equivalence of their Gagliardo-Slobodeckij norms to norms defined via Littlewood-Paley decompositions, oscillations, or Besov type difference quotients is established as well. This equivalence, of independent interest, is a key tool in the proof of the relevant embeddings. They also   rest upon a new optimal inequality for convolutions in Orlicz spaces.
\end{abstract}

\maketitle

%


%
%
%
%
%
%
%
%



\setcounter{tocdepth}{1}

\unmarkedfootnote{
}

%

\section{Introduction}\label{intro}
The fractional Orlicz-Sobolev spaces $W^{s,A}(\rn)$ of order $s\in (0,1)$ are defined via a Gagliardo-Slobodeckij seminorm of Luxemburg type built upon a Young function $A$. Spaces with fractional smoothness $s\in (1, \infty) \setminus \mathbb N$ can accordingly be defined by requiring that their members be functions whose weak derivatives of order $[s]$, the integer part of $s$, have a finite Orlicz-Sobolev seminorm of order $\{s\}$, the fractional part of $s$.

This family of spaces includes the classical fractional Sobolev spaces $W^{s,p}(\rn)$, for  $p\in [1,\infty)$, which are reproduced with the choice $A(t)=t^p$.  The spaces  $W^{s,p}(\rn)$ are a natural functional setting for nonlocal elliptic problems with power-type nonlinearities. When dealing with the same kind of problems, but driven by non-polynomial nonlinearities, the more general spaces $W^{s,A}(\rn)$ provide a suitable ambient for the solutions. 
 The study of the spaces $W^{s,A}(\rn)$ has been initiated in recent years, and various traits of their theory have already been shaped - see e.g. \cite{ACPS_lim0, ACPS_limit1, ACPS, BaS, Bon:19, deNBS, KKPF}.

Embedding theorems for the spaces $W^{s,A}(\rn)$ are central in view of the applications mentioned above. Embeddings into Banach spaces endowed with norms depending only on (global) integrability properties of functions were established in \cite{ACPS, ACPSinf}. Among other results, the optimal Orlicz target space $L^{A_{\frac ns}}(\rn)$ for embeddings of homogeneous fractional Orlicz-Sobolev spaces associated with any admissible $s$ and $A$ was detected in those papers. The embedding
\begin{equation}\label{intoOrlicz}
W^{s,A}(\rn) \to L^{A_{\frac ns}}(\rn)
\end{equation}
 for the non-homogeneous spaces $W^{s,A}(\rn)$ follows from them. The Young function $A_{\frac ns}$   takes an explicit form depending only on  $A$ and on the ratio $\frac ns$.

On the other hand, embeddings between Orlicz-Sobolev spaces with different smoothness, of the form 
\begin{equation}\label{embintro}
    W^{s,A}(\rn) \to W^{r,B}(\rn),
\end{equation}
where $0<r<s$, seem to be still missing the literature. Our purpose is to fill in this gap and provide an analogue of embedding \eqref{intoOrlicz}, which associates with any admissible  $A$, $r$ and $s$  an appropriate conjugate Young function $B$ making embedding \eqref{embintro} true. The main result of this paper asserts that embedding \eqref{embintro} holds with  
$$B=A_{\frac n{s-r}}.$$ 
 Namely, the function $B$ in embedding \eqref{embintro} is given exactly according to the same rule as the function  $A_{\frac ns}$ in \eqref{intoOrlicz}, save that $s$ has to be replaced by the smoothness gap $s-r$. Let us emphasize that, like the result for \eqref{intoOrlicz}, our conclusions about embeddings \eqref{embintro} do not require restrictions on the behavior of the Young function $A$, such as the $\Delta_2$ or the $\nabla_2$ condition. 
In particular when  $A(t)=t^p$, we recover the customary embedding
$
    W^{s,p}(\rn) \to W^{r, \frac{np}{n-(s-r)p}}(\rn),
$
where  $0<s-r<n$ and $1\leq p <\frac n{s-r}$. Embeddings for fractional Zygmund-Sobolev spaces, namely Orlicz-Sobolev spaces associated with Young functions of \lq\lq power times a logarithm" type  are also deduced as special cases of our general result.
%

Despite the perfect formal matching of embeddings \eqref{intoOrlicz} and \eqref{embintro} for $r=0$, 
 their proofs are completely different. Gagliardo-Slobodeckij type seminorms do not seem suitable to deal with embeddings between fractional Orlicz-Sobolev spaces of different smoothness. A substantial part of this paper is thus devoted to proving the equivalence of these norms to more manageable norms of diverse types. Precisely,  Besov type norms, norms involving integral oscillations, and norms based on Littlewood-Paley decompositions are considered. 
Such an equivalence is classically well-known for the spaces $W^{s,p}(\rn) $, but was
not available for the more general spaces $W^{s,A}(\rn)$. Interestingly, unlike the conventional $W^{s,p}(\rn)$ theory, borderline situations do not require separate treatment in the Orlicz-Sobolev ambient. For instance, the case when $A$ is infinite-valued,  and, in particular, when  $L^A(\rn)=L^\infty(\rn)$, is included in our discussion.  The occurrence of arbitrary Young functions prevents us from using standard tools, such as approximation by smooth functions which need not be dense in the spaces in question, at several steps in the proofs, and hence calls for the introduction of novel strategies.


The embedding theorem is stated in Section \ref{main}, while definitions of the new norms and their equivalences are the subject of Section \ref{equivnorm}. 
Proofs  these equivalences are offered in Sections \ref{sec:besov}--\ref{sec:LP}. 
In particular, the alternative Littlewood-Paley characterization of the spaces $W^{s,A}(\rn)$ is pivotal in our approach to Theorem \ref{thm:main}, which is accomplished in Section \ref{sec:emb}. Another critical ingredient is an optimal inequality for convolutions in Orlicz spaces, of independent interest for possible different applications. The latter improves a classical convolution inequality by O'Neil from \cite{oneil2}, whose use would not lead to optimal embeddings in general. Our convolution inequality  
 is proved in Section \ref{sec:conv}, that also collects basic definitions and properties of Young functions and Orlicz spaces. 

\section{Embeddings between fractional Orlicz-Sobolev spaces}\label{main}
 A function $A: [0, \infty) \to [0, \infty]$ is called a
Young function if it is convex, left-continuous, vanishing at $0$, and neither
identically equal to $0$, nor to $\infty$. 
The Luxemburg norm associated with  $A$  is defined
as
\begin{align*}
\|u\|_{L^A(\rn)}=\inf\left\{\lambda>0:\,\,\int_{\rn}
A\Big(\frac{|u(x)|}{\lambda}\Big)\,dx\leq 1\right\}
\end{align*}
for any function $u \in \mathcal M(\rn)$, the space of real-valued measurable functions in $\rn$. The
collection of all  functions $u$ for which such norm is finite is
called the  Orlicz space $L^A(\rn)$, and  is a Banach function
space. 
\par Given $s\in (0,1)$, the seminorm $|u|_{s,A, \rn}$ of  a function $u \in \mathcal M (\rn)$ is given by
\begin{equation}\label{aug340}
|u|_{s,A, \rn}
		= \inf\left\{\lambda>0: \int_{\rn} \int_{\rn}A\left(\frac{|u(x)-u(y)|}{\lambda|x-y|^s}\right)\frac{\dd\,(x,y)}{|x-y|^n}\le1\right\}\,.
\end{equation}
Notice that, if $|u|_{s,A, \rn}<\infty$, then $u\in L^1_{\rm loc}(\rn)$.
\\
The fractional-order Orlicz-Sobolev space $W^{s,A}(\rn)$ is then defined as
\begin{equation}\label{aug344}
W^{s,A}(\rn ) = \{ u \in L^A(\rn ):\, |u|_{s,A, \rn} <\infty\},
\end{equation}
and is a   Banach space equipped with the norm
$$\|u\|_{W^{s,A}(\rn)} = \|u\|_{L^A(\rn)} + |u|_{s,A, \rn}.$$

The definition of  the space $W^{s,A}(\rn)$ 
for $s\in (1, \infty) \setminus \N$ requires the notion of integer-order Orlicz-Sobolev space $W^{m,A}(\rn )$  associated with 
  $m \in \N$.
The latter is defined as
\begin{equation}\label{orliczsobolev}
W^{m,A}(\rn ) = \{u \in L^A(\rn): |\nabla ^k u| \in L^A(\rn), \, k=1, \dots, m\}\,.
\end{equation}
Here, $\nabla ^k u$ denotes the vector of all weak derivatives of $u$ of order $k$. If $m=1$, we also simply write $\nabla u$ instead of $\nabla^1 u$.
The space $W^{m,A}(\rn)$ is a Banach space equipped with the
norm
$$\|u\|_{W^{m,A}(\rn )} = \sum _{k=0}^m  \|\nabla^k u
\|_{L^A(\rn)},$$
where $\nabla ^0u$ has to be interpreted as $u$.
\\ Now, given $s\in (1, \infty) \setminus \N$, 
the fractional-order Orlicz-Sobolev space $W^{s,A}(\rn )$ is defined by
\begin{equation}\label{aug344higher}
W^{s,A}(\rn ) = \{ u \in W^{[s],A}(\rn): \big|\nabla ^{[s]}u\big|_{\{s\},A, \rn}<\infty\},
\end{equation}
and is endowed with the norm
$$\|u\|_{W^{s,A}(\rn)} = \|u\|_{W^{[s],A}(\rn)} + \big|\nabla ^{[s]}u\big|_{\{s\},A, \rn}.$$
 
The formulation of our embedding between Orlicz-Sobolev spaces associated with different smoothness parameters $s$ requires the following notions.
\\
Let $r,s \in (0, \infty) \setminus \N$, with   $0<s-r<n$, and let $A$  be a Young function  such that
\begin{align}\label{dec1}
\int_0 \Big(\frac{t}{A(t)}\Big)^{\frac{s-r}{n-s+r}}\dt<\infty.
\end{align}
We define the Young function $A_{\frac{n}{s-r}}$ as
\begin{align}\label{dec2}
A_{\frac{n}{s-r}}(t)=A(H^{-1}(t))\quad \text{for $t\geq0$,}
\end{align}
where
\begin{align}\label{dec2'}
H(t)=\bigg(\int_0^t\Big(\frac{\tau}{A(\tau)}\Big)^\frac{s-r}{n-s+r}\,\mathrm{d}\tau\bigg)^{\frac{n-s+r}{n}} \quad \text{for $t\geq0$.}
\end{align}
If $\lim_{t \to \infty}H(t)<\infty$, then 
 $H^{-1}$ has to be interpreted as the generalized left-continuous inverse of $H$. Therefore, $H^{-1}(t)=\infty$, and hence $A_{\frac{n}{s-r}}(t)=\infty$, if $t> \int_0^\infty\big(\frac{\tau}{A(\tau)}\big)^\frac{s-r}{n-s+r}\,\mathrm{d}\tau$.
%

\begin{theorem}\label{thm:main}
Let $r,s \in (0, \infty) \setminus \N$, with $0<s-r<n$.  Assume that the Young function $A$  satisfies condition \eqref{dec1}, and let $A_{\frac{n}{s-r}}$ be the Young function defined by \eqref{dec2}. Then  
\begin{align}\label{eq:thmmain}
W^{s, A}(\R^{n})\to W^{r, A_{\frac{n}{s-r}}}(\R^{n}).
\end{align}
Moreover, there exists a constant $c=c(n, s, r)$ such that
\begin{equation}\label{normineq}
\|u\|_{W^{r, A_{\frac{n}{s-r}}}(\R^{n})} \leq  c \|u\|_{ W^{s, A}(\R^{n})}
\end{equation}
for every $u \in W^{s, A}(\R^{n})$.
\end{theorem}

\begin{remark}\label{integralform}
{\rm    Inequality \eqref{normineq} can also be stated in an integral form, which is possibly handier in some applications. Indeed, since the constant $c$ in \eqref{normineq} is independent of $A$, this inequality can be applied with $A(t)$ replaced by the function $A_M(t) = A(t)/M$, where 
\begin{equation}\label{modular}
M = \sum _{k \leq [s]}\int_{\rn}A  \big(  |\nabla ^k  u| \big)\dx +  \int_{\rn} \int_{\rn}A  \left(\frac{|\nabla^{[s]}u(x)-\nabla^{[s]}u(y)|}{|x-y|^{\{s\}}}\right)\frac{\dd\,(x,y)}{|x-y|^n}.
\end{equation}
One can verify that the function $(A_M)_{\frac{n}{s-r}}$, associated with $A_M$ as in \eqref{dec2}--\eqref{dec2'}, obeys  $(A_M)_{\frac{n}{s-r}}(t)=  M^{-1} A_{\frac{n}{s-r}}( tM^{-\frac{s-r}n})$. Hence, from the definition of Luxemburg norms, one infers that 
\begin{equation}\label{intineq}
\sum _{k \leq [r]}\int_{\rn}A_{\frac{n}{s-r}} \bigg( \frac{ |\nabla ^k  u|}{c M^{\frac{s-r}n}} \bigg)\dx +  \int_{\rn} \int_{\rn}A_{\frac{n}{s-r}} \left(\frac{|\nabla^{[r]}u(x)-\nabla^{[r]}u(y)|}{c|x-y|^{\{r\}}M^{\frac{s-r}n}}\right)\frac{\dd\,(x,y)}{|x-y|^n} \leq M
\end{equation}
for $u \in W^{s, A}(\R^{n})$.}
\end{remark}
%

We conclude this section with a couple of applications of Theorem  \ref{thm:main} to fractional Orlicz-Sobolev spaces built upon Young functions $A$ of a special form. The former deals with Young functions $A$  of power type, which is possibly modified near $0$ in the supercritical regime in such a way that condition \eqref{dec1} is satisfied. The latter is a generalization of the former, in that Young functions $A$ of \lq\lq power times a logarithm" type are considered.

\begin{example}\label{ex01}
{\rm Let $r,s \in (0, \infty) \setminus \N$, with $0<s-r<n$.
\\ (i) Assume that $1\leq p  <\frac n{s-r}$ and 
\begin{equation}\label{jan60}
A(t) = t^p \qquad \text{for $t \geq 0$.}
\end{equation}
Then
$$
A_{\frac{n}{s-r}}(t) \approx t^{\frac{np}{n-(s-r)p}} \qquad \text{for $t \geq 0$.}
$$
Here, and in what follows, the relation $\lq\lq \approx"$ means equivalence in the sense of Young functions.
Hence, Theorem \ref{thm:main} reproduces the classical embedding 
$$W^{s,p}(\rn) \to W^{r, \frac{np}{n-(s-r)p}}(\rn).$$
(ii) Assume that the Young function $A$ is such that
\begin{equation}\label{jan61}
A(t) = \begin{cases} t^p & \quad \text{near infinity}
\\ t^{p_0} & \quad \text{near zero,}
\end{cases}
\end{equation}
where   $ p  \geq \frac n{s-r}$ and $1 \leq p_0<\frac n{s-r}$. 
\\ If  $ p  >\frac n{s-r}$, then 
\begin{equation}\label{jan62}
A_{\frac{n}{s-r}}(t) \approx \begin{cases} \infty & \quad \text{near infinity}
\\ t^{\frac{np}{n-(s-r)p_0}} & \quad \text{near zero.}
\end{cases}
\end{equation}
\\ If  $ p  = \frac n{s-r}$, then 
\begin{equation}\label{jan62'}
A_{\frac{n}{s-r}}(t)  \approx \begin{cases} e^{t^{\frac n{n-(s-r)}}} & \quad \text{near infinity}
\\ t^{\frac{np}{n-(s-r)p_0}} & \quad \text{near zero.}
\end{cases}
\end{equation}
In particular, in the latter case, a fractional embedding of Pohozaev-Trudinger-Yudovich type into an exponential space is obtained.
}
  
\end{example}

\begin{example}\label{ex02}{\rm
Let $0<r<s<\infty$, with $0<s-r<n$.
Consider   a Young function $A$ such that
\begin{equation}\label{dec250}
A(t) \approx \begin{cases}
t^p  (\log t)^\alpha & \quad \text{near infinity}
\\
 t^{p_0} (\log \frac 1t)^{\alpha_0} & \quad \text{near zero,}
\end{cases}
\end{equation}
where: 
$$\text{either $p>1$ and $\alpha \in \R$ or $p=1$ and $\alpha \geq 0$,}$$
and
\begin{equation*}
\text{either $1< p_0< \frac n{s-r}$ and $\alpha_0\in \R$, or $p_0=1$ and  $\alpha_0 \leq 0$, or $p_0=\frac n{s-r}$ and $\alpha_0 > \frac n{s-r} -1$.}
\end{equation*}
%
%
%
%
%
%
Theorem \ref{thm:main} then tells us that embedding  \eqref{eq:thmmain}  holds,  where
\begin{equation}\label{dec255}
A_{\frac n{s-r}}(t) \approx \begin{cases} t^{\frac {n{p_0}}{n-(s-r){p_0}}} (\log \frac 1t)^{\frac {n\alpha_0}{n-(s-r){p_0}}} & \quad \text{ if $1\leq {p_0}< \frac n{s-r}$ }
\\
e^{-t^{-\frac{n}{(s-r)(\alpha_0 +1)-n}}} & \quad  \text{if ${p_0}=\frac n{s-r}$ and $\alpha_0 > \frac n{s-r} -1$}
\end{cases} \quad \text{near zero,}
\end{equation}
and
\begin{equation}\label{dec256}
A_{\frac n{s-r}}(t) \approx \begin{cases} t^{\frac {np}{n-(s-r)p}} (\log t)^{\frac {n \alpha  }{n-(s-r)p}} & \quad  \text{ if $1\leq p< \frac n{s-r}$ }
\\
e^{t^{\frac{n}{n-(s-r)(\alpha +1)}}}&  \quad \text{if  $p=\frac n{s-r}$ and $\alpha < \frac n{s-r} -1$}
\\
e^{e^{t^{\frac n{n-(s-r)}}}} &  \quad \text{if  $p=\frac n{s-r}$ and $\alpha = \frac n{s-r} -1$}
\\ \infty &  \quad \text{otherwise}
\end{cases} \quad \text{near infinity.}
\end{equation}
 
}

\end{example}

\section{Equivalent norms}\label{equivnorm}

In this section, we collect definitions and results about alternate norms in the space $W^{s,A}(\rn)$. Norms based on Littlewood-Paley decompositions are exploited in the proof of Theorem \ref{thm:main}. Their equivalence to the Gagliardo-Slobodeckij norm is in turn intertwined with the proof of the equivalence to Besov and oscillation-type norms.

\subsection{Norms of Besov type}\label{sub:besov}

Given $s\in (0,1)$ and a Young function $A$, define the Besov type seminorm  $|u|_{B^{s,A}(\rn)}$ of  a function $u \in \mathcal M (\rn)$   as  

\begin{equation}\label{besov1}
|u|_{B^{s,A}(\rn)}= \inf\bigg\{\lambda >0: 
\sum_{j=1}^{n}\int_{0}^{\infty}\int_{\R^{n}}A\Big(\frac{|u(x+\varrho e_{j})-u(x)|}{\lambda \,\varrho^s}\Big)\dx\frac{\dd \varrho}{\varrho}\leq 1\bigg\}.
\end{equation}
Here, $e_j$ denotes the $j$-th unit vector of the canonical orthonormal basis in $\rn$.
\\ The Besov space $B^{s,A}(\rn)$  is then defined as 
\begin{equation}\label{besov2}
B^{s,A}(\rn) = \{u \in L^A(\rn):  |u|_{B^{s,A}(\rn)}<\infty\},
\end{equation}
and is equipped with the norm given by
\begin{equation}\label{besov3}
\|u\|_{B^{s,A}(\rn)}  = \|u\|_{L^A(\rn)}+ |u|_{B^{s,A}(\rn)}.
\end{equation}
If $s\in (1, \infty) \setminus \mathbb  N$, then the space $B^{s,A}(\rn)$  is accordingly defined as
\begin{equation}\label{besov4}
B^{s,A}(\rn) = \big\{u \in W^{[s],A}(\rn):  |\nabla ^{[s]}u|_{B^{{\{s\}, A}}(\rn)}<\infty\big\},
\end{equation}
and is endowed with the norm
\begin{equation}\label{besov5}
\|u\|_{B^{s,A}(\rn)}  = \|u\|_{W^{[s],A}(\rn)}+ |\nabla ^{[s]}u|_{B^{\{s\}, A}(\rn)}.
\end{equation}

\begin{theorem}\label{besovequiv}
Let  $s\in (0,\infty)\setminus \mathbb N$ and let $A$ be a Young function. Then,
\begin{equation}\label{besov7}
B^{s,A}(\rn) = W^{s,A}(\rn),
\end{equation} 
 with norms equivalent up to constants depending on $n$.
\end{theorem}

\bigskip
\subsection{Norms defined via oscillations}\label{sub:osc}

Given $s \in (0, \infty) \setminus \mathbb N$, we set
\begin{align*}
\osc^s u(x,r)=\inf_{\mathfrak q\in\mathcal P_{[ s]}}\dashint_{B_r(x)}\frac{|u-\mathfrak q|}{r^s}\dy,
\end{align*}
for $u \in\mathcal M(\rn)$.
Here, $B_r(x)$ stands for the ball centered at $x$ and with radius $r$, and $\mathcal P_{[s]}$ denotes the space of polynomials of degree at most $[ s]$.
With  a Young function $A$, we associate the seminorm given by
\begin{align*}
|u|_{\mathcal O^{s,A}(\R^n)}=\inf \bigg\{\lambda >0 :\int_0^1\int_{\R^n}A\Big(\frac{\osc^s u(x,r)}{\lambda}\Big)\dx\frac{\dd r}{r}\leq 1\bigg\}
\end{align*}
for $u \in\mathcal M(\rn)$.
\\
The space $\mathcal O^{s,A}(\R^n)$ is then defined as 
\begin{equation}\label{osc1}
\mathcal O^{s,A}(\R^n)= \{u\in L^A(\rn): |u|_{\mathcal O^{s,A}(\R^n)}<\infty\},
\end{equation}
and is equipped with the norm given by
\begin{equation}\label{osc2}
\|u\|_{\mathcal O^{s,A}(\R^n)}= \|u\|_{L^A(\rn)} +  |u|_{\mathcal O^{s,A}(\R^n)}
\end{equation}
for $u \in\mathcal M(\rn)$.
%
\begin{theorem}\label{thmosc}
Let $s \in (0, \infty) \setminus \mathbb N$ and let $A$ be a Young function.
Then,
\begin{equation}\label{osc3} \mathcal O^{s,A}(\rn) = W^{s,A}(\rn),
\end{equation}
 with norms equivalent up to constants depending on $n$ and $s$.
%
%
\end{theorem}

\bigskip
\subsection{Norms defined via  a Littlewood-Paley decomposition}\label{sub:fourier}

Let $\varphi_{0}\in C_{0}^{\infty}(B_2(0))$ be such that 
$\varphi_{0}=1$ in $B_1(0)$.
 We define for $i\in\mathbb{N}$ the function $\varphi_{i} : \rn \to \mathbb R$ as
\begin{align}
\varphi_{i}(\xi)=\varphi_{0}(2^{-i}\xi)-\varphi_{0}(2^{-i+1}\xi)\quad \text{for $\xi\in\R^{n}$.} 
\end{align}
Then: 
\begin{enumerate}
\item $\mathrm{sprt}(\varphi_{0})\subset\overline{B_2(0)}$.
\item\label{item:supporitem} $\mathrm{sprt}(\varphi_{i})\subset \{\xi\in\R^{n}\colon\;2^{i-1}\leq |\xi|\leq 2^{i+1}\}$ for all $i\in\mathbb{N}$.
\item If $i,l\in\mathbb{N}$ satisfy $|i-l|\geq 2$, then $\mathrm{sprt}(\varphi_{i})\cap \mathrm{sprt}(\varphi_{l})=\emptyset$. 
\item The sequence $\{\varphi_{i}\}_{i=0}^{\infty}$ forms a partition of unity in $\rn$, namely
\begin{align}
\sum_{i=0}^{\infty}\varphi_{i}(\xi)=1\quad\text{for}\;\xi\in\R^{n}. 
\end{align}
\end{enumerate}
Consider the Fourier multiplying operator
$\varphi_i(D)$ given by 
\begin{align*}
\varphi_i(D)u=\mathfrak F^{-1}(\varphi_i\widehat u)
\end{align*}
for a tempered distribution $u$. 
Then  
\begin{align*}
u
=\sum_{i=0}^\infty\varphi_i(D)u.
\end{align*}
Given a Young function $A$ and $s\in (0, \infty) \setminus \N$, we define $F^{s,A}(\R^n)$ as the Banach space   of all tempered distributions $u$ for which the  norm
\begin{align*}
\|u\|_{F^{s,A}(\R^n)}= \inf \bigg\{\lambda>0 : \sum_{i=0}^\infty \int_{\R^n}A\Big(\frac{2^{is}|\varphi_i(D)u|}{\lambda}\Big)\dx\leq 1\bigg\}
\end{align*}
is finite.  

\begin{theorem}\label{thm:equi}
Let   $s\in (0, \infty) \setminus \mathbb N$ and let $A$ be a Young function. Then,
\begin{equation}\label{fourier1} 
F^{s,A}(\rn) = W^{s,A}(\rn),
\end{equation}
  with norms equivalent up to constants depending on $n$ and $s$.
\end{theorem}

\section{Convolutions in Orlicz spaces}\label{sec:conv}

This section is devoted to a convolution inequality in Orlicz spaces which is critical in the proof of  Theorem \ref{thm:main}. We begin by recalling a few definitions and basic properties from the
theory of Young functions and Orlicz spaces. For a
comprehensive treatment of this topic, we refer to
\cite{RR1,RR2}.
\par
With any Young
function $A$,  it is uniquely associated a (nontrivial) non-decreasing
left-continuous function $a:[0, \infty) \rightarrow [0, \infty]$
such that
\begin{equation}\label{B.4}
A(t) = \int _0^t a(r)\,dr \qquad {\rm for} \,\, t\geq 0.
\end{equation}
 If $A$ is any Young function and $\lambda \geq 1$, then
\begin{equation}\label{lambdaA}
\lambda A(t) \leq A(\lambda t) \quad \hbox{for $t \geq 0$.}
\end{equation}
As a consequence, if $\lambda \geq 1$, then
\begin{equation}\label{lambdaA-1}
 A^{-1}(\lambda r) \leq \lambda A^{-1}(r) \quad \hbox{for $r \geq 0$.}
\end{equation}
\par\noindent
A Young function $A$ is said to dominate another Young function $B$
[near infinity]\, [near zero] if there exists a positive constant  $C$ such that
\begin{equation}\label{B.5bis}
B(t)\leq A(c t) \qquad \textrm{for \,\,\,$t\geq 0$\,\, [$t\geq t_0$
\,\, for some $t_0>0$]\,\, [$0\leq t\leq t_0$
\,\, for some $t_0>0$]\,.}
\end{equation}
The functions $A$ and $B$ are called equivalent [near infinity] \, [near zero] if
they dominate each other [near infinity]\, [near zero]. We shall write $A \approx
B$ to denote this equivalence.
 \\  The notion of domination between Young functions enables one to characterize embeddings between Orlicz spaces. 
 Let $A$ and $B$ be Young functions. Then,
\begin{equation}\label{B.6}
L^A(\rn)\to L^B(\rn),
\end{equation}
if and only if $A$ dominates $B$ globally. The norm
of the embedding \eqref{B.6} depends on the
 constant $c$ appearing in
\eqref{B.5bis}.
\\
The Young conjugate $\widetilde{A}$ of a Young function $A$ is the Young function
defined by
$$\widetilde{A}(t) = \sup \{rt - A(r):\,r\geq 0\} \qquad {\rm for}\qquad  t\geq 0\,.$$
Note the representation formula
$$\widetilde A (t) = \int _0^t a^{-1}(r)\,dr \qquad {\rm for} \,\, t\geq
0,$$ where $a^{-1}$ denotes the (generalized) left-continuous
inverse of $a$. 
One has that
\begin{equation}\label{AAtilde}
r \leq A^{-1}(r)\widetilde A^{-1}(r) \leq 2 r \quad \hbox{for $r
\geq 0$,}
\end{equation}
where $A^{-1}$ denotes the (generalized) right-continuous inverse of
$A$.
Moreover,
\begin{equation}\label{B.7'}
\widetilde{\!\widetilde A\,}=A\,
\end{equation}
for any Young function $A$.
\\
The    Young conjugate enters 
a H\"older type inequality in Orlicz spaces.  
One has that
\begin{equation}\label{B.5}
\|v\|_{L^{\widetilde{A}}(\rn)}\leq \sup _{u \in L^A(\rn)}
\frac {\int _{\rn} u(x) v(x)\, dx}{\|u\|_{L^A(\rn)}} \leq 2
\|v\|_{L^{\widetilde{A}}(\rn)}
\end{equation}
for every $v\in L^{\widetilde{A}}(\rn)$.
\\ In our proofs we shall also   make use of Orlicz sequence spaces. The  Orlicz sequence space built upon a Young function $A$ is defined as the space $\ell^A(\mathbb Z)$  of all sequences $\mathfrak a= \{a_i\}$ for which the norm
$$\|\mathfrak a\|_{\ell^A(\mathbb Z)} = \inf\bigg\{\lambda >0: \sum_{i\in \mathbb Z}A\Big(\frac {|a_i|}\lambda\Big)\leq 1\bigg\}$$
is finite.

 A basic convolution inequality, involving a kernel that merely belongs to $L^1(\rn)$, is stated in the next theorem.

\begin{theorem}\label{lem:cont0}
Let $A$ be a Young function. Then,
\begin{align}\label{eq:1903}
\|u*v\|_{L^A(\R^n)}\leq \|v\|_{L^1(\R^n)}\|u\|_{L^A(\R^n)}
\end{align}
for every $u\in L^A(\R^n)$ and $v\in L^{1}(\R^n)$.
 Moreover,
\begin{align}\label{eq:1903'}
\int_{\R^n}A(|u*v|)\dx\leq \int_{\R^n}A\Big(\|v\|_{L^{1}(\R^n)}|u|\Big)\dx.
\end{align}
\end{theorem}

A discrete counterpart of Theorem \ref{lem:cont0} for Orlicz sequence spaces is the subject of the next result. 

\begin{theorem}\label{lem:discrete} Let $A$ be a Young function. Then,
\begin{equation}\label{dec6}
\|\mathfrak a * \mathfrak b\|_{\ell^A(\mathbb Z)} \leq \|\mathfrak a\|_{\ell^1(\mathbb Z)} \|\mathfrak b\|_{\ell^A(\mathbb Z)}
\end{equation}
for all sequences $\mathfrak a \in \ell^1(\mathbb Z)$ and $\mathfrak b \in \ell^A(\mathbb Z)$. 
Moreover,
\begin{align}\label{dec5}
\sum_{l\in\mathbb Z} A\bigg(\Big|\sum_{{  i\in \mathbb Z}}a_{i-l}b_l\Big|\bigg)\leq \sum_{l\in\mathbb Z} A\big(\|\mathfrak a\|_{\ell^1(\mathbb Z)}|b_l|\big),
\end{align}
where $\mathfrak a= \{a_i\}$ and $\mathfrak b= \{b_i\}$. 
\end{theorem}

We limit ourselves to proving Theorem \ref{lem:discrete}, the proof of Theorem
\ref{lem:cont0}  being just a continuous version of the former.

\begin{proof}[Proof of Theorem \ref{lem:discrete}] Inequality \eqref{dec5} is a consequence of the following chain:
\begin{align}\label{dec6'}
\sum_{l\in\mathbb Z} A\bigg(\Big|\sum_{{  i\in \mathbb Z}}a_{i-l}b_l\Big|\bigg)  &\leq 
\sum_{l\in\mathbb Z} A\bigg( \sum_{{  i\in \mathbb Z}}|a_{i-l}||b_l|\bigg) 
 =  \sum_{l\in\mathbb Z} A\Bigg(\frac{ \sum_{{  i\in \mathbb Z}}|a_{i-l}| \|\mathfrak a\|_{\ell^1(\mathbb Z)}  |b_l|}{ \|\mathfrak a\|_{\ell^1(\mathbb Z)} }\Bigg) 
\\\nonumber  & \leq   \sum_{l\in\mathbb Z}   \sum_{i\in\mathbb Z} A\big( \|\mathfrak a\|_{\ell^1(\mathbb Z)}|b_l|\big)\frac{|a_{i-l}| }{ \|\mathfrak a\|_{\ell^1(\mathbb Z)} }
 =\sum_{l\in\mathbb Z}  A\big( \|\mathfrak a\|_{\ell^1(\mathbb Z)}|b_l|\big) \sum_{i\in\mathbb Z} \frac{|a_{i-l}| }{ \|\mathfrak a\|_{\ell^1(\mathbb Z)} }\nonumber
\\  & = \sum_{l\in\mathbb Z}  A\big( \|\mathfrak a\|_{\ell^1(\mathbb Z)}|b_l|\big).
\nonumber
\end{align}
Note that the inequality holds thanks to a discrete version of Jensen's inequality.
\\ By the definition of Luxemburg norms,  inequality \eqref{dec6} can be deduced via an application of inequality \eqref{dec5} with $ \{a_i\}$ and $ \{b_i\}$  replaced by $\{a_i/\|\mathfrak a\|_{\ell^1(\mathbb Z)}\}$ and $\{b_i/ \|\mathfrak b\|_{\ell^A(\mathbb Z)}\}$.
\end{proof}

Theorem \ref{lem:cont} below provides us with an inequality for convolution operators whose  kernel belongs to a Lebesgue space different from $L^1(\rn)$.
Definition \eqref{dec2} of the Young function $A_{\frac n{s-r}}$ defining the fractional Orlicz-Sobolev target space in the embedding of Theorem \ref{lem:cont0} is dictated by this convolution inequality. The function $A_{\frac n{s-r}}$ is associated with
the sharp target Orlicz space for the convolution operator with a domain space $L^A(\rn)$ and a kernel in $L^{\frac n{n-s+r}}(\rn)$. 

For simplicity of notation, with a shift of exponents, we state the relevant convolution inequality for kernels in $L^{\frac n{n-s}}(\rn)$. This entails the use of a Young function defined as in \eqref{dec2}, but with $s-r$ replaced just by $s$. We reproduce an explicit definition of the resultant function $A_{\frac ns}$ for the reader's convenience.

Let   $s\in(0,n)$ and let $A$ be a a Young function such that  
\begin{align}\label{eq:A2}
\int_0 \Big(\frac{t}{A(t)}\Big)^{\frac{s}{n-s}}\dt<\infty.
\end{align}
We define the Young function
\begin{align}\label{eq:Ans}
A_{\frac{n}{s}}(t)=A(H^{-1}(t))\quad \text{for $t\geq0$,}
\end{align}
where
\begin{align*}
H(t)=\bigg(\int_0^t\Big(\frac{\tau}{A(\tau)}\Big)^\frac{s}{n-s}\,\mathrm{d}\tau\bigg)^{\frac{n-s}{n}} \quad \text{for $t\geq0$,}
\end{align*}
and $H^{-1}$ stands for the generalized left-continuous inverse of $H$. Note that this is the classical inverse if 
\begin{align}\label{eq:A1}
\int^\infty \Big(\frac{t}{A(t)}\Big)^{\frac{s}{n-s}}\dt=\infty.
\end{align}
On the other hand, if this condition is not fulfilled, namely if $A$ grows so fast near infinity that
\begin{align}\label{eq:A3}
\int^\infty \Big(\frac{t}{A(t)}\Big)^{\frac{s}{n-s}}\dt<\infty,
\end{align}
then $H^{-1}(t) = \infty$ if $t> t_\infty$, where
\begin{equation}\label{tinfinity}
t_\infty = \bigg(\int_0^\infty \Big(\frac{t}{A(t)}\Big)^{\frac{s}{n-s}}\dt\bigg)^{\frac{n-s}{n}}.
\end{equation}
In this case, equation \eqref{eq:Ans} has also  to be interpreted as 
\begin{align}\label{eq:Ansinf}
A_{\frac{n}{s}}(t)=\infty \quad \text{for $t >t_\infty$.}
\end{align}

\begin{theorem}\label{lem:cont}
Let  $s\in(0,n)$. Assume that $A$ is  a Young function satisfying condition \eqref{eq:A2}  and let  $A_{\frac{n}{s}}$ be the Young function defined by  \eqref{eq:Ans}. Then, there exists a constant $c=c(n,s)$ such that
\begin{align}\label{eq:1903a}
\|u*v\|_{L^{A_{\frac{n}{s}}}(\R^n)}\leq\,c\,\|v\|_{L^{\frac{n}{n-s}}(\R^n)}\|u\|_{L^A(\R^n)}
\end{align}
for   $u\in L^A(\R^n)$ and $v\in L^{\frac{n}{n-s}}(\R^n)$.
\\ Moreover,
\begin{align}\label{conv5}
\int_{\R^n}A_{\frac{n}{s}}\Bigg(\frac{|u*v|}{c \|v\|_{L^{\frac{n}{n-s}}(\R^n)} (\int_{\R^n}A(|u|)\dy)^{s/n}}
\Bigg)\dx\leq \int_{\R^n}A(|u|)\dx.
\end{align}
\end{theorem}

A proof of Theorem \eqref{lem:cont} makes use of rearrangements of functions and of norms in Lorentz 
spaces. Recall that the decreasing rearrangement of a function $u \in \mathcal M(\rn)$ is the unique non-increasing right-continuous function $u^*: [0, \infty) \to [0, \infty]$ equimeasurable with $u$. Namely,
$$u^*(t)= \inf\{t\geq 0: |\{|u|>\tau\}|\leq t\} \quad \text{for $t \geq 0$.}$$
The function $u^{**}:  (0, \infty) \to [0, \infty]$,  defined as $u^{**}(t)= \frac 1t \int_0^tu^*(r)\, dr$ for $t>0$,  is also non-increasing.
\\ Lebesgue norms and, more generally, Luxemburg norms are rearrangement-invariant, in the sense that
$$\|u\|_{L^A(\rn)} = \|u^*\|_{L^A(0,\infty)}$$
for  $u \in \mathcal M(\rn)$.   Lorentz norms are defined in terms of rearrangements, and hence trivially enjoy the same property. In particular, the norm in the Lorentz space $L^{p,1}(\rn)$, with $p>1$, is given by
$$\|u\|_{L^{p,1}(\rn)} = \int_0^\infty u^{*}(t)t^{-1+\frac 1p}\, dt$$
for  $u \in \mathcal M(\rn)$.

\begin{proof}[Proof of  Theorem  \ref{lem:cont}] Let $v\in L^{\frac{n}{n-s}}(\R^n)$. By inequality \eqref{eq:1903}, with $L^A(\rn)= L^{\frac{n}{n-s}}(\rn)$, one has that
\begin{equation}\label{conv3}
\|u*v\|_{L^{\frac{n}{n-s}}(\R^n)} \leq  \|v\|_{L^{\frac{n}{n-s}}(\rn) }\|u\|_{{L^1}(\R^n)}
\end{equation}
if $u \in L^1(\R^n)$. Moreover, 
an application of H\"older's inequality and the rearrangement-invariance of the norm in $L^{\frac{n}{n-s}}(\R^n)$ imply that
\begin{equation}\label{conv2}
t^{\frac{n-s} n} v^{**}(t) \leq  \|v\|_{L^{\frac{n}{n-s}}(\R^n)} \quad \text{for $t\geq 0$.}
\end{equation}
Thus, by O'Neil's rearrangement inequality for convolutions (see \cite{oneil}), 
\begin{equation}\label{conv1}
(u*v)^{**}(t) \leq \|v\|_{L^{\frac{n}{n-s}}(\R^n)}\bigg( t^{\frac{s-n} n}\int_0^t u^* (r) \dr + \int_t^\infty u^* (r) r^{\frac{s-n} n} \dr\bigg) \quad \text{for $t>0$.}
\end{equation}
Taking the limit in inequality \eqref{conv1} as $t\to 0^+$  tells us that
\begin{align}\label{conv6}
 \|u*v\|_{L^{\infty}(\R^n)}   = \lim_{t\to 0^+} (u*v)^{**}(t)  &\leq   \|v\|_{L^{\frac{n}{n-s}}(\R^n)} \int_0^\infty u^* (r) r^{\frac{s-n} n} \dr \\ \nonumber & =  \|v\|_{L^{\frac{n}{n-s}}(\R^n)} \|u\|_{L^{\frac ns, 1}(\R^n)}
\end{align}
if $u \in L^{\frac ns, 1}(\R^n)$. 
 Notice that, in deriving  equation \eqref{conv6}, we have made use of the fact that
$$\lim_{t\to 0^+} t^{\frac{s-n} n}\int_0^t u^* (r) \dr \leq \lim_{t\to 0^+}  \int_0^t u^* (r)r^{\frac{s-n} n} \dr =0.$$
Owing to inequalities \eqref{conv3} and \eqref{conv6}, for each fixed function $v\in L^{\frac{n}{n-s}}(\R^n)$, the linear operator $T_v$, given by
$$T_v(u) = u*v,$$
satisfies the endpoint boundedness properties:
\begin{equation}\label{bound}
T_v : L^1(\R^n) \to  L^{\frac{n}{n-s}}(\R^n)\quad \text{and} \quad T_v: L^{\frac ns, 1}(\R^n) \to L^\infty (\R^n),
\end{equation}
 with norms not exceeding $\|v\|_{L^{\frac{n}{n-s}}}$.
\\ Assume first that condition \eqref{eq:A1} is fulfilled. Then, thanks to properties \eqref{bound} of the operator $T_v$, inequality \eqref{conv5} follows from the interpolation result of  \cite[Theorem 4]{Cianchi_CPDE}. 
\\ Suppose next that condtion \eqref{eq:A3} is in force. By inequality \eqref{B.5},
$$ \int_0^\infty u^* (r) r^{\frac{s-n} n} \dr\leq 2  \|u\|_{L^A(\rn)} \|r^{\frac{s-n} n}\|_{L^{\widetilde A}(0, \infty)}$$
if $u \in L^A(\rn)$. Here, we have exploited the rearrangement-invariance of the norm in $L^A(\rn)$.
Computations show that
\begin{equation}\label{conv7}
\|r^{\frac{s-n} n}\|_{L^{\widetilde A}(0, \infty)} = \bigg(\frac n{n-s} \int_0^\infty \frac{\widetilde A(t)}{t^{1+\frac n{n-s}}}\, \dt\bigg)^{\frac{n-s}n},
\end{equation}
see e.g. \cite[Equation (3.10)]{cianchi_ASNS}. Hence, inequality \eqref{conv6} yields
\begin{align}\label{conv8}
 \|u*v\|_{L^{\infty}(\R^n)}   \leq   \|v\|_{L^{\frac{n}{n-s}}(\R^n)}\bigg(\frac n{n-s} \int_0^\infty \frac{\widetilde A(t)}{t^{1+\frac n{n-s}}}\, \dt\bigg)^{\frac{n-s}n} \|u\|_{L^A(\R^n)}.
\end{align}
Set $M=  \int_{\R^n}A(|u|)\dx$ and consider the Young function $A_M$ given by $A_M(t) = \frac{A(t)}M$ for $t\geq 0$. Notice that 
$$\widetilde {A_M}(t)= \frac 1M \widetilde A (Mt) \quad \text{for $t \geq 0$.}$$
Since
$$ \|u\|_{L^{A_M}(\R^n)} \leq 1,$$
an application of inequality \eqref{conv8} with $A$ replaced by $A_M$ yields, after a change of variable in the integral,
\begin{align}\label{conv9}
 \|u*v\|_{L^{\infty}(\R^n)}  & \leq   \|v\|_{L^{\frac{n}{n-s}}(\R^n)}\bigg(\frac n{n-s} \int_0^\infty \frac{\widetilde A(t)}{t^{1+\frac n{n-s}}}\, \dt\bigg)^{\frac{n-s}n} M^{\frac sn} 
\\ \nonumber & =     \|v\|_{L^{\frac{n}{n-s}}(\R^n)}\bigg(\frac n{n-s} \int_0^\infty \frac{\widetilde A(t)}{t^{1+\frac n{n-s}}}\, \dt\bigg)^{\frac{n-s}n} \bigg(  \int_{\R^n}A(|u|)\dx \bigg)^{\frac sn}.
\end{align}
Next, the following chain holds:
\begin{align}\label{conv10}
\int_0^\infty \frac{\widetilde A(t)}{t^{1+\frac n{n-s}}}\, \dt & \leq \int_0^\infty \frac{t a^{-1}(t)}{t^{1+\frac n{n-s}}}\, \dt = \int_0^\infty t^{-\frac n{n-s}}\int_0^{a^{-1}(t)}\dr\, \dt 
\\ \nonumber &=  \int_0^\infty \int_{a(r)}^\infty t^{-\frac n{n-s}}\, \dt\, \dr 
= \frac{n-s}s \int_0^\infty \frac 1{a(r)^{\frac s{n-s}}}\, \dr
\\ \nonumber & \leq \frac{n-s}s \int_0^\infty \bigg(\frac r {A(r)}\bigg)^{\frac s{n-s}}\, \dr =  \frac{n-s}s\, (t_\infty)^{\frac n{n-s}},
\end{align}
where $t_\infty$ is defined equation \eqref{tinfinity}.
Coupling inequality \eqref{conv9} with \eqref{conv10} implies that
\begin{align}\label{conv11}
 \|u*v\|_{L^{\infty}(\R^n)}    \leq    t_\infty\, \Big(\frac ns\Big)^{\frac {n-s}n} \|v\|_{L^{\frac{n}{n-s}}(\R^n)} \bigg(  \int_{\R^n}A(|u|)\dx \bigg)^{\frac sn}.
\end{align}
Hence, 
\begin{align}\label{conv12}
\frac  {|u*v(x)|}{ \big(\tfrac ns\big)^{\frac {n-s}n} \|v\|_{L^{\frac{n}{n-s}}(\R^n)} \Big(  \int_{\R^n}A(|u|)\dx \Big)^{\frac sn}} \leq   t_\infty \quad \text{for a.e. $x \in \R^n$.}
\end{align}
Therefore, on replacing, if necessary, the constant $c$ in inequality \eqref{conv5} by the constant $c'=\max\{c, 2\big(\tfrac ns\big)^{\frac {n-s}n}\}$, one may assume that inequality \eqref{conv12} is fulfilled in the proof of \eqref{conv5}. If this is the case, one can verify that, in the proof of   \cite[Theorem 4]{Cianchi_CPDE}, the function $H^{-1}$ is always evaluated in the interval $[0, t_\infty)$. Since, in this interval, it coincides with the classical inverse of $H$, the same proof of  \cite[Theorem 4]{Cianchi_CPDE} applies to deduce \eqref{conv5} also in this case.
\\ Finally, inequality \eqref{eq:1903a} follows, via the definition of Luxemburg norm, via an application of inequality \eqref{conv5} with $u$ replaced by $\frac{u}{\|u\|_{L^A(\R^n)}}$.
\end{proof}

\section{Proof of Theorem \ref{besovequiv}}\label{sec:besov}
 
This section is devoted to the proof of the equivalence of Gagliardo-Slobodecki norms and Besov type norms in fractional Orlicz-Sobolev spaces. In this proof,  and  in what follows, by $c$ and $c_i$, with $i\in \mathbb N$, we denote constants that may only depend on the quantities specified in the statement of the result to be proved and may vary from one equation to another.

\begin{proof}[Proof of Theorem \ref{besovequiv}] First, consider the case when $s\in(0,1)$. In order to prove equation \eqref{besov7}, it   suffices to show that  there exist constants $c_1$ and $c_2$, depending  on $n$, such that
\begin{equation}\label{besov6}
c_1 |u|_{B^{s,A}(\rn)}  \leq   |u|_{s,A,\rn}  \leq c_2 |u|_{B^{s,A}(\rn)}
\end{equation}
for every $u\in \mathcal M(\rn)$. Given  such a function $u$, the use of 
 polar coordinates yields
\begin{align*}
\iint_{\R^{n}\times\R^{n}}A\Big(\frac{|u(x)-u(y)|}{|x-y|^{s}}\Big)&\,\frac{\dd\,(x,y)}{|x-y|^n}  = \iint_{\R^{n}\times\R^{n}}A\Big(\frac{|u(x+h)-u(x)|}{|h|^{s}}\Big)\dx\,\frac{\dd h}{|h|^n}\\
& = \int_{0}^{\infty}\int_{\R^{n}}\int_{\partial B_1(0)}A\Big(\frac{|u(x+\varrho\omega)-u(x)|}{\varrho^s}\Big)\,\dd\omega\dx\frac{\dd \varrho}{\varrho}
\end{align*}
To estimate the latter integral from above,
we write $\omega$ as a linear combination of the vectors $e_j$ and use the elementary inequality
\begin{align*}
|u(x+\varrho\omega)-u(x)| & \leq \sum_{j=1}^{n}|u(x+\varrho\omega_{1}e_{1}+...+\varrho\omega_{j}e_{j})-u(x+\varrho\omega_{1}e_{1}+...+\varrho\omega_{j-1}e_{j-1})|
\end{align*}
for suitable $\omega_j \in [-1, 1]$,  $ j=1, \dots , n$,
where we have set $e_0=0$. Owing to this inequality and a change of coordinates,  
\begin{align*}
\int_{0}^{\infty}\int_{\R^{n}}&\int_{\partial B_1(0)}A\Big(\frac{|u(x+\varrho\omega)-u(x)|}{\varrho^s}\Big)\,\dd\omega\dx\frac{\dd \varrho}{\varrho}\\
& \leq \sum_{j=1}^{n}\int_{0}^{\infty}\int_{\partial B_1(0)}\int_{\R^{n}}A\Big(c\frac{|u(x+\varrho\omega_j e_{j})-u(x)|}{\varrho^s}\Big)\,\dd\omega\dx\frac{\dd \varrho}{\varrho}\\
&=\sum_{j=1}^{n}\int_{\partial B_1(0)}\int_{\R^{n}}\int_{0}^{\infty}A\Big(c\frac{|u(x+\varrho\omega_j e_{j})-u(x)|}{\varrho^s}\Big)\frac{\dd \varrho}{\varrho}\,\dd\omega\dx.
\end{align*}
For $\omega\in\partial B_1(0)$ and $j\in\{1,\dots,n\}$, we have that
\begin{align*}
\int_{0}^{\infty}\int_{\R^{n}}A\Big(c\frac{|u(x+\varrho\omega_j e_{j})-u(x)|}{\varrho^s}\Big)\dx\frac{\dd \varrho}{\varrho}&=\int_{0}^{\infty}\int_{\R^{n}}A\Big(c|\omega_j|^s\frac{|u(x+\varrho e_{j})-u(x)|}{\varrho^s}\Big)\dx\frac{\dd \varrho}{\varrho}\\
&\leq\int_{0}^{\infty}\int_{\R^{n}}A\Big(c\frac{|u(x+\varrho e_{j})-u(x)|}{\varrho^s}\Big)\dx\frac{\dd \varrho}{\varrho}.
\end{align*}
Notice that this equality holds  trivially  if $\omega_j=0$, whereas if $\omega_j<0$ it follows via a preliminary change of variable from $x$ into $ x-\varrho\omega_j e_{j}$.
Altogether, we conclude that
\begin{align}\label{eq:trivial}
\iint_{\R^{n}\times\R^{n}}&A\Big(\frac{|u(x)-u(y)|}{|x-y|^{s}}\Big)\,\frac{\dd\,(x,y)}{|x-y|^n} \leq \sum_{j=1}^{n}\int_{0}^{\infty}\int_{\R^{n}}A\Big(c\frac{|u(x+\varrho e_{j})-u(x)|}{\varrho^s}\Big)\dx\frac{\dd \varrho}{\varrho}.
\end{align}
We now prove the reverse inequality. To fix ideas, assume that $j=n$.  Given $x \in \rn$, set $x=(x', x_n)$ with $x'\in\R^{n-1}$, and
\begin{align*}
Q_{r}(x')=\bigtimes_{j=1}^{n-1}(x_j-r,x_j+r).
\end{align*}
Since $A$ is an increasing function,
\begin{align*}
A\Big(\frac{|u(x+\varrho e_{n})-u(x)|}{\varrho^s}\Big)
&\leq \dashint_{Q_{\varrho/2}(x')}A\Big(2\frac{|u(x',x_n+\varrho )-u(y',x_n+\varrho/2)|}{\varrho^s}\Big)\,\dd y'\\
& \quad +\dashint_{Q_{\varrho/2}(x')}A\Big(2\frac{|u(y',x_n+\varrho/2 )-u(x',x_n)|}{\varrho^s}\Big)\,\dd y',
\end{align*}
where  $\dashint_E=\frac 1{|E|}\int_E$,  the averaged integral over the set $E$.
Hence,
\begin{align}\label{march300}
\int_{0}^{\infty}\int_{\R^{n}}&A\Big(\frac{|u(x+\varrho e_{n})-u(x)|}{\varrho^s}\Big)\dx\frac{\dd \varrho}{\varrho}\\ \nonumber &\leq \int_0^\infty\int_{\R}\int_{\R^{n-1}} \dashint_{Q_{\varrho/2}(x')}A\Big(2\frac{|u(x',x_n+\varrho )-u(y',x_n+\varrho/2)|}{\varrho^s}\Big)\,\dd y'\,\dd x'\,\dd x_n\frac{\dd \varrho}{\varrho}\\ \nonumber 
&\quad +\int_0^\infty\int_{\R}\int_{\R^{n-1}}\dashint_{Q_{\varrho/2}(x')}A\Big(2\frac{|u(y',x_n+\varrho/2 )-u(x',x_n)|}{\varrho^s}\Big)\,\dd y'\,\dd x'\,\dd x_n\frac{\dd \varrho}{\varrho}.
\end{align}
One has that
\begin{align}\label{march301}
 \int_0^\infty&\int_{\R}\int_{\R^{n-1}} \dashint_{Q_{t/2}(x')}A\Big(2\frac{|u(x',x_n )-u(y',x_n-\varrho/2)|}{\varrho^s}\Big)\,\dd y'\,\dd x'\,\dd x_n\frac{\dd \varrho}{\varrho}\\ \nonumber 
&=\int_0^\infty\int_{\R}\int_{\R^{n-1}}\int_{\R^{n-1}} \chi_{Q_{\varrho/2}(x')}(y')A\Big(2\frac{|u(x',x_n )-u(y',x_n-\varrho/2)|}{\varrho^s}\Big)\,\dd y'\,\dd x'\,\dd x_n\frac{\dd \varrho}{\varrho^{n}}\\ \nonumber
&=\int_0^\infty\int_{\R^{n-1}}\int_{\R^{n-1}}\int_{\R} \chi_{Q_{\varrho/2}(x')}(y')A\Big(2\frac{|u(x',x_n )-u(y',x_n-\varrho/2)|}{\varrho^s}\Big)\,\dd x_n\,\dd y'\,\dd x'\frac{\dd \varrho}{\varrho^{n}}\\ \nonumber
&=\int_0^\infty\int_{\R^{n-1}}\int_{\R^{n-1}}\int_{\R} \chi_{(2|y'-x'|_\infty,\infty)}(\varrho)A\Big(2\frac{|u(x',x_n )-u(y',x_n-\varrho/2)|}{\varrho^s}\Big)\,\dd x_n\,\dd y'\,\dd x'\frac{\dd \varrho}{\varrho^{n}},
\end{align}
where $\chi_E$ stands for the characteristic function of the set $E$.
Here, we have made use of the fact  that $y'\in Q_{\varrho/2}(x')$ if and only if $\varrho>2|y'-x'|_\infty,$ where $|\cdot|_\infty$ denotes the maximum norm in $\R^{n-1}$. 
Since
$$|(x',x_n )-(y',x_n-\varrho/2)|\leq \,c\varrho$$
for some constant $c=c(n)$ and for $\varrho\in (2|y'-x'|_\infty,\infty)$, we   deduce  from equation \eqref{march301} the following bound for the first addend on the right-hand side of inequality \eqref{march300}:
\begin{align*}
\int_0^\infty& \int_{\R}\int_{\R^{n-1}} \dashint_{Q_{\varrho/2}(x')}A\Big(2\frac{|u(x',x_n+\varrho )-u(y',x_n+\varrho/2)|}{\varrho^s}\Big)\,\dd y'\,\dd x'\,\dd x_n\frac{\dd \varrho}{\varrho}
\\ \nonumber &\leq \int_0^\infty\int_{\R^{n-1}}\int_{\R^{n-1}}\int_{2|y'-x'|_\infty}^\infty A\Big(2\frac{|u(x',x_n )-u(y',x_n-\varrho/2)|}{|(x',x_n )-(y',x_n-\varrho/2)|^s}\Big)\frac{\dd \varrho\,\dd x_n\,\dd y'\,\dd x'}{|(x',x_n )-(y',x_n-\varrho/2)|^n}\\
&\leq \int_{\R}\int_{\R^{n-1}}\int_{\R^{n-1}}\int_{\R} A\Big(2\frac{|u(x',x_n )-u(y',y_n)|}{|(x',x_n )-(y',y_n)|^s}\Big)\frac{\dd y_n\dd x_n\,\dd y'\,\dd x'}{|(x',x_n )-(y',y_n)|^n}\\
&=\iint_{\R^{n}\times\R^{n}}A\Big(2\frac{|u(x)-u(y)|}{|x-y|^{s}}\Big)\,\frac{\dd\,(x,y)}{|x-y|^n}.
\end{align*}
An analogous argument tells us that the second addend on the right-hand side of inequality \eqref{march300} admits the same bound.
In conclusion, since, of course,  $e_n$ can be replaced by any $e_j$ with $j=1, \dots , n-1$, we have shown that  
\begin{align}
\label{eq:2402}
\sum_{j=1}^{n}\int_{0}^{\infty}&\int_{\R^{n}}A\Big(\frac{|u(x+\varrho e_{j})-u(x)|}{\varrho^s}\Big)\dx\frac{\dd \varrho}{\varrho} \leq  
\iint_{\R^{n}\times\R^{n}}A\Big(c\frac{|u(x)-u(y)|}{|x-y|^{s}}\Big)\,\frac{\dd\,(x,y)}{|x-y|^n}.
\end{align}
Equation \eqref{besov6} is thus proved.
\\ When $s>1$,  from equation \eqref{besov6} applied with $s$ replaced by $\{s\}$ and $u$ by $\nabla ^{[s]}u$, one obtains that
\begin{equation}\label{besov6bis}
c_1 |\nabla ^{[s]}u|_{B^{\{s\},A}(\rn)}  \leq   |\nabla ^{[s]}u|_{\{s\},A,\rn}  \leq c_2 |\nabla ^{[s]}u|_{B^{\{s\},A}(\rn)}
\end{equation}
for $B^{s,A}(\rn)$.
Equation \eqref{besov7} is   a consequence of  \eqref{besov6bis}.
\end{proof}

\section{Equivalent  norms defined via oscillations}\label{sec:osc}

 Here we focus on the equivalence of Gagliardo-Slobodecki norms and norms of oscillation type.
This is the content   of Theorem \ref{thmosc}, whose proof requires a couple of preliminary lemmas.
 
\begin{lemma}\label{lem:BsA}
Let $s\in (0,\infty)\setminus \N$. If $u\in F^{s,A}(\R^n)$,  then $u$ admits weak derivatives of order $[s]$. Moreover,   $\nabla^{[ s]}u\in L^A(\R^n)$ and $\nabla^{[ s]}u\in F^{\{s\},A}(\R^n)$, and  there exist positive constants $c_1$ and $c_2$, depending on $n$ and $s$, such that
\begin{align}\label{eq:0603}
  \int_{\R^n}A(c_1|\nabla ^{[s]}u|)\dx\leq 
\sum_{i=0}^\infty\int_{\R^n}&A\big(2^{i\{s\}}\big|\varphi_i(D)\nabla^{[ s]} u\big|\big)\dx\leq \sum_{i=0}^\infty \int_{\R^n}A(c_2 \,2^{is}|\varphi_i(D)u|)\dx
\end{align}
for every $u\in F^{s,A}(\R^n)$.
\end{lemma}
\begin{proof}
 We begin by proving the first inequality in \eqref{eq:0603} for $s\in(0,1)$.  Notice that, for  $s$ in this range, the second inequality trivially holds as an identity with $c_2=1$. We may assume, without loss of generality, that 
$u\in F^{s,A}(\R^n)$ is such that  
\begin{align}\label{apr40}
\sum_{i=0}^\infty \int_{\R^n}A(2^{is}|\varphi_i(D)u|)\dx<\infty.
\end{align}
 By the convexity of the function $A$, for every $m \in \N$,
\begin{align}\label{jan80}
\int_{\R^n}A\bigg(\frac 1{\sum _{i=0}^\infty 2^{-is}}\sum_{i=0}^m|\varphi_i(D)u|\bigg)\dx&\leq 
\int_{\R^n}A\bigg(\frac 1{\sum _{i=0}^m 2^{-is}}\sum_{i=0}^m|\varphi_i(D)u|\bigg)\dx
\\ \nonumber & \leq 
 \sum_{i=0}^m  \frac{2^{-is}}{\sum _{i=0}^m 2^{-is}}  \int_{\R^n} A\big( 2^{si} |\varphi_i(D)u|\big)\dx
\\ \nonumber &\leq 
 \sum_{i=0}^\infty \int_{\R^n}  A\big( 2^{si} |\varphi_i(D)u|\big)\dx
<\infty.
\end{align}
Hence, by Fatou's Lemma, the series $\sum_{i=0}^\infty|\varphi_i(D)u|$, and hence  $\sum_{i=0}^\infty\varphi_i(D)u$, converge a.e. in $\rn$ and 
\begin{align}\label{jan80'}
\int_{\R^n}A\bigg(\bigg|(1-2^{-s})\sum_{i=0}^\infty\varphi_i(D)u\bigg|\bigg)\dx&\leq 
\int_{\R^n}A\bigg((1-2^{-s})\sum_{i=0}^\infty|\varphi_i(D)u|\bigg)\dx
\\ \nonumber & \leq 
 \sum_{i=0}^\infty \int_{\R^n}A(2^{si}|\varphi_i(D)u|)\dx.
\end{align}
This tells us that $u \in L^A(\rn)$ and  inequality  \eqref{eq:0603} holds for $s\in(0,1)$.
\\
Next, we show that, if $s\in(1,\infty)\setminus \N$ and $u\in F^{s,A}(\R^n)$, 
then $u$ is 
$[ s]$-times weakly differentiable. Assume, without loss of generality,  that $u$ fulfills inequality \eqref{apr40}. For a multi-index $\alpha=(\alpha_1,\dots,\alpha_n)\in\N^n_0$ with $|\alpha|=\alpha_1+\dots\alpha_n=[ s]$ we set $\xi^\alpha=\xi_1^{\alpha_1}\dots\xi_n^{\alpha_n}$. 
Set, for simplicity, 
\begin{equation}\label{march318}
u_i=\varphi_i(D)u,
\end{equation}
whence
 $u=\sum_{i=0}^\infty u_i$.  We claim that the function $u_i$ is  $[s]$-times weakly differentiable for every $i$.  Clearly, to prove that the weak derivative $\partial^\alpha u_i$ exists and coincides with $\mathfrak F^{-1}(  (\mathfrak i  \, \cdot)^\alpha \widehat{u}_i)$ it suffices to show that the latter is a function in $L^1_{\rm loc}(\R^n)$. Here, $\mathfrak i $ denotes the imaginary unit. Let $\eta \in C^\infty_0(\rn)$  be a function such that, on setting  $\eta_i(x)=\eta(2^{-i}x)$, we have that $\eta_i=1$ in ${\rm spt} \varphi _i$. Then
\begin{align*}
\mathfrak F^{-1}( (\mathfrak i \,\cdot)^\alpha \widehat{u}_i)=\mathfrak F^{-1}( (\mathfrak i\, \cdot)^\alpha\varphi_i\widehat u)=\mathfrak F^{-1}\Big( \frac{(\mathfrak i \,\cdot)^\alpha\eta_i}{2^{si}}\varphi_i 2^{si}\widehat{u}\Big)=\mathfrak F^{-1}\Big(  \frac{(\mathfrak i \, \cdot)^\alpha\eta_i}{2^{si}}\Big)*\mathfrak F^{-1}\big(\varphi_i 2^{si}\widehat{u}\big).
\end{align*}
For each $i\in\N$,  one has that $\eta_i\in\mathscr S(\R^n)$.  Therefore, $  \frac{(\mathfrak i\, \cdot)^\alpha\eta_i}{2^{si}}\in\mathscr S(\R^n)$ and hence 
$\mathfrak F^{-1}\big( \frac{(\mathfrak i\, \cdot)^\alpha\eta_i}{2^{si}}\big)\in \mathscr S(\R^n)\subset L^1(\R^n)$. Estimate \eqref{eq:1903} yields 
\begin{align*}
\|\mathfrak F^{-1}((\mathfrak i \, \cdot)^\alpha \widehat{u}_i)\|_{L^A(\R^n)}\leq\,c\,\|\mathfrak F^{-1}\big(\varphi_i 2^{si}\widehat{u}\big)\|_{L^A(\R^n)}<\infty
\end{align*}
for some constant $c$ independent of $u$.
Consequently, $\mathfrak F^{-1}( (\mathfrak i \cdot)^\alpha \widehat{u}_i)\in L^A(\R^n)\subset L^1_{\rm loc}(\R^n)$. In particular, 
$u_i\in W^{[ s],A}(\R^n)$.  Moreover, by inequality  \eqref{eq:1903'}
and the identity
\begin{align}\label{march310}
\partial^\alpha u_i=\mathfrak F^{-1}(  (\mathfrak i\, \cdot)^\alpha \widehat{u}_i)=\mathfrak F^{-1}((\mathfrak i \, \cdot)^\alpha \varphi_i\widehat{u})=\mathfrak F^{-1}\big(\varphi_i\widehat{\partial^\alpha u}\big)=\varphi_i(D)\partial^\alpha u,
\end{align}
which holds for every {\color{red} $i\geq 0$} and $\alpha\in \N^n_0$ with $|\alpha|\leq [ s]$,
we also have that
\begin{align}\label{jan82}
\sum_{i=0}^\infty\int_{\R^n}&A\big(c 2^{\{s\}i}\big|\partial ^\alpha u_i\big|\big)\dx =
\sum_{i=0}^\infty\int_{\R^n}A\big(c 2^{\{s\}i}\big|\big(\mathfrak F^{-1}( (\mathfrak i\, \cdot)^\alpha \widehat{u}_i)\big|\big)\dx
\\ \nonumber &=  \sum_{i=0}^\infty
\int_{\R^n}A\big(c 2^{\{s\}i+[s]i}\big| 
\mathfrak F^{-1}\big((2^{-i}\cdot)^\alpha\eta(2^{-i}\cdot)\big)
*\mathfrak F^{-1}\big(\varphi_i \widehat{u}\big)\big|\big)
\big)\dx
\\ & \nonumber
 \leq\,\sum_{i=0}^\infty\int_{\R^n}A\big(2^{si}\big|\mathfrak F^{-1}\big(\varphi_i \widehat{u}\big)\big|\big)\dx 
 = \,\sum_{i=0}^\infty\int_{\R^n}A\big(2^{si}|\varphi_i(D) u|\big)\dx.
\end{align}
 Here we have used the fact that
\begin{align*}
\big\|\mathfrak F^{-1}\big((2^{-i}\cdot)^\alpha\eta(2^{-i}\cdot)\big)\big\|_{L^1(\R^n)}=\big\|2^{in}\mathfrak F^{-1}\big((\cdot)^\alpha\eta\big)(2^i\cdot)\big\|_{L^1(\R^n)}=\big\|\mathfrak F^{-1}\big((\cdot)^\alpha\eta\big)\big\|_{L^1(\R^n)}
\leq c
\end{align*}
for $i \geq 0$.
  Thanks to equations \eqref{march310} and \eqref{jan82}, an application of inequalities \eqref{jan80} and \eqref{jan80'} with $u$ replaced by  $\partial^ \alpha u_i$ and $s$ by $\{s\}$ tells us that $u\in W^{[ s],A}(\R^n)$ and
\begin{align}\label{jan81}\int_{\R^n} A(c_1 |\nabla ^{[s]}u|)\dx& \leq \sum_{|\alpha|= [s]}
\int_{\R^n} A\bigg(\bigg|c_2 \sum_{i=0}^\infty \partial^ \alpha u_i\bigg|\bigg)\dx
\\ \nonumber & \leq 
  \sum_{|\alpha|= [s]}\sum_{i=0}^\infty \int_{\R^n}A(c2^{\{s\}i} |\partial ^\alpha u_i|)\dx \\ \nonumber & \leq\,\sum_{i=0}^\infty\int_{\R^n}A\big(2^{si}|\varphi_i(D) u|\big)\dx
\end{align}
Inequality  \eqref{eq:0603} hence follows.
\end{proof}

The next lemma shows how the functions
$\varphi_i$ in the definition of $F^{s,A}(\R^n)$ can be replaced by functions
whose inverse Fourier transform has compact support. In fact, we only prove one of the two estimates required to prove that such a replacement results in an equivalent definition of $F^{s,A}(\R^n)$ (up to equivalent norms). The proof of the missing estimate is quite lengthy and not needed for our purposes (see e.g. \cite[Theorem 2.4.1]{Tr2} for the case of power type functions $A$).
\begin{lemma}\label{lem:support}
Let $A$ be a Young function and let $s>0$.
Assume that the function $\mathfrak K_0\in C^\infty_0(\R^n)$ satisfies the inequality
\begin{align}\label{eq:1904}
\mathfrak F(\mathfrak K_0){\color{black}>} 0\quad\text{in}\,\, B_2(0).
\end{align}
Given $N\in\N$, set $\mathfrak K= \Delta^N\mathfrak K_0$, where $\Delta$ denotes the Laplace operator.  
Then, there exists a positive constant   $c=c(s,n, \mathfrak K_0)$
 such that
\begin{align}\label{march313}
\int_{\R^n}A(2^{is}|\varphi_i(D)u|)\dx\leq \int_{\R^n}A(c2^{is}|2^{in}\mathfrak K(2^{i}\cdot)*u|)\dx,
\end{align}
for  $i\geq 1$ and  $u \in L^A(\rn)$.
\end{lemma}
\begin{proof}  Let   $u\in L^A(\R^n)$.
Thanks to assumption \eqref{eq:1904}, we have that $\widehat{\mathfrak K}>0$ in $B_{2}(0)\setminus B_{1/2}(0)$. Hence,  
\begin{align}\label{eq:1904B}
\widehat{\mathfrak K}(2^{-i}\cdot)>0\quad\text{in}\quad B_{2^{i+1}}(0)\setminus B_{2^{i-1}}(0)\supset {\rm sprt}(\varphi _i).
\end{align}
Therefore,
\begin{align*}
\int_{\R^n}A(2^{is}|\varphi_i(D) u|)\dx&=
\int_{\R^n}A\Big(2^{is}\Big|\mathfrak F^{-1}\Big(\frac{\varphi_i}{\widehat{\mathfrak K}(2^{-i}\cdot)}\widehat{\mathfrak K}(2^{-i}\cdot)\widehat{ u}\Big)\Big|\Big)\dx\\
&=\int_{\R^n}A\Big(2^{is}\Big|\mathfrak F^{-1}\Big(\frac{\varphi_i}{\widehat{\mathfrak K}(2^{-i}\cdot)}\Big)*\mathfrak F^{-1}\Big(\widehat{\mathfrak K}(2^{-i}\cdot)\widehat{u}\Big)\Big|\Big)\dx
\end{align*}
for $i \geq 1$.
{\color{black} By setting $\psi(\xi) =\varphi_0(\xi)-\varphi(2\xi)$}, we have that
\begin{align}\label{march312}
\bigg\|\mathfrak F^{-1}\Big(\frac{\varphi_i}{\widehat{\mathfrak K}(2^{-i}\cdot)}\Big)\bigg\|_{L^1(\R^n)}&=2^{in}\bigg\|\mathfrak F^{-1}\Big(\frac{\psi}{\widehat{\mathfrak K}}\Big)(2^{i}\cdot)\bigg\|_{L^1(\R^n)}=\bigg\|\mathfrak F^{-1}\Big(\frac{\psi}{\widehat{\mathfrak K}}\Big)\bigg\|_{L^1(\R^n)},
\end{align}
{\color{black} for $i\geq 1$}.
Since $\widehat{\mathfrak K}$ is strictly positive in the support of {\color{black} $\psi$},  the function $\psi/\widehat{\mathfrak K}$  belongs to $\mathscr S(\R^n)$, and hence also $\mathfrak F^{-1}\big(\psi/\widehat{\mathfrak K}\big)$ does. Therefore, the last norm in equation \eqref{march312} is finite, and
 Lemma \ref{lem:cont0} yields
\begin{align*}
\int_{\R^n}A(2^{is}|\varphi_i(D) u|)\dx&\leq \int_{\R^n}A\big(c2^{is}\big|\mathfrak F^{-1}\big(\widehat{\mathfrak K}(2^{-i}\cdot)\widehat{u}\big)\big|\big)\dx.
\end{align*}
Hence  inequality \eqref{march313} follows, inasmuch as 
 $\mathfrak F^{-1}\big(\widehat{\mathfrak K}(2^{-i}\cdot)\widehat{u}\big)=2^{in}\mathfrak K(2^{i}\cdot)*u$.  Note that the latter equality holds in the sense of distributions, owing to a standard property of the inverse Fourier transform.   By our assumption on $u$ and Theorem \ref{lem:cont0},  we have that $2^{in}\mathfrak K(2^{i}\cdot)*u \in L^A(\rn)$. Therefore,  $\mathfrak F^{-1}\big(\widehat{\mathfrak K}(2^{-i}\cdot)\widehat{u}\big) \in L^A(\rn)$ as well.
%
%
\end{proof}

\begin{proof}[Proof of Theorem \ref{thmosc}]  We shall prove that 

\begin{align}\label{eq:2903A'}
\int_{\R^n}A(c|u|)\dx&+\int_0^1\int_{\R^n}A\big(\osc^s (cu)(x,r)\big)\dx\frac{\dd r}{r}\leq \sum_{i=0}^\infty\int_{\R^n} A(2^{si}|\varphi_i(D)u|)\dx
\end{align}
for some constant $c=c(n,s)$ and  for $u \in F^{s,A}(\rn)$, and 
\begin{align}\label{eq:2903A''}
\sum_{i=0}^\infty\int_{\R^n} A(2^{si}|\varphi_i(D)u|)\dx \leq \int_{\R^n}A(c|u|)\dx+\int_0^1\int_{\R^n}A\big(c\osc^s (u)(x,r)\big)\dx\frac{\dd r}{r}
\end{align}
for   some constant $c=c(n,s)$ and  $u \in \mathcal O^{s,A}(\rn)$. Coupling inequality \eqref{eq:2903A'} with   \eqref{eq:2903A''} yields:
\begin{equation}\label{apr50}
\mathcal O^{s,A}(\R^n) = F^{s,A}(\R^n),
%
\end{equation}
with norms equivalent up to multiplicative constants depending on $n$ and $s$.
Equation \eqref{osc3} will then follow from Theorem \ref{thm:equi}, whose proof,  given in the next section, in turn, relies upon some inequalities to be established here.
\\ 
First, consider inequality \eqref{eq:2903A'}. 
Fix  $u\in F^{s,A}(\R^n)$.  Lemma \ref{lem:BsA} ensures that $u \in L^A(\rn)$. 
Set
\begin{align*}
\osc^s_A u(x,r) =\inf_{\mathfrak q\in\mathcal P_{[ s]}}\dashint_{B_r(x)}A\Big(\frac{|u-\mathfrak q|}{r^s}\Big)\dy.
\end{align*}
By Jensen's inequality,
\begin{align}\label{jan50}
\int_0^1\int_{\R^n}A\big(\osc^s u(x,r)\big)\dx\frac{\dd r}{r}&\leq \int_0^1\int_{\R^n}\osc^s_A u(x,r)\dx\frac{\dd r}{r}\\ \nonumber
&{\color{black} \leq } \sum_{l=0}^\infty\int_{\R^n}\int_{2^{-l-1}}^{2^{-l}}\osc^s_A (c_1u)(x,2^{-l})\frac{\dd r}{r}\dx\\ \nonumber
&\leq \sum_{l=0}^\infty\int_{\R^n}\osc^s_A (cu)(x,2^{-l})\dx.
\end{align}
First, assume that $s\in(0,1)$.  Jensen's inequality again entails that
\begin{align*}
\osc^s_A (cu)(x,2^{-j})
&\leq\dashint_{B_{2^{-j}}(x)}\dashint_{B_{2^{-l}}(x)}A\big(c2^{js}|u(y)-u(z)|\big)\dz\dy\\
&=\dashint_{B_{2^{-l}}(0)}\dashint_{B_{2^{-l}}(0)}A\big(c2^{ls}|u(x+y)-u(x+z)|\big)\dz\dy,
\end{align*}
whence
\begin{align*}
\int_{\R^n}\osc^s_A (cu)(x,{\color{black} 2^{-j}})\dx
&\leq\dashint_{B_{2^{-j}}(0)}\dashint_{B_{2^{-l}}(0)}\int_{\R^n}A\big(c2^{ls}|u(x+y)-u(x+z)|\big)\dx\dz\dy\\
&=\dashint_{B_{2^{-l}}(0)}\dashint_{B_{2^{-l}}(0)}\int_{\R^n}A\big(c2^{js}|u(x)-u(x+z-y)|\big)\dx\dz\dy\\
&\leq \sup_{|h|\leq 2^{-l+1}}\int_{\R^n}A\big(c2^{ls}|u(x)-u(x+h)|\big)\dx.
\end{align*}
Altogether,
\begin{align}\label{march315}
\int_0^1\int_{\R^n}\osc^s_A u(x,r)\dx\frac{\dd r}{r}&\leq\sum_{l=0}^\infty\sup_{|h|\leq 2^{-l+1}}\int_{\R^n}A\big(c2^{ls}|u(x)-u(x+h)|\big)\dx\\ \nonumber
&\leq\sum_{l=1}^\infty\sup_{|h|\leq 2^{-l}}\int_{\R^n}A\big(c2^{ls}|u(x)-u(x+h)|\big)\dx\\ \nonumber & \quad +\sup_{|h|\leq 2}\int_{\R^n}A\big(c|u(x)-u(x+h)|\big)\dx.
\end{align}
Clearly, the second addend on the rightmost side of equation \eqref{march315} is bounded by $\int_{\R^n}A\big(c|u|\big)\dx$, and, in turn, the latter is estimated (up to a constant in the argument of $A$) by the right-hand side of inequality \eqref{eq:2903A'}, thanks to the first inequality in \eqref{eq:0603}.
 \\ On the other hand,  the first addend is bounded by
\begin{align*}
\sum_{j=1}^n\sum_{l=1}^\infty \sup_{t\leq 2^{-l}}\int_{\R^{n}}A\Big(c 2^{ls}|u(x+t e_{j})-u(x)|\Big)\dx.
\end{align*}
Now, we claim that
%
\begin{align}\label{eq:2903}
\begin{aligned}
\sum_{l=1}^\infty \sup_{t\leq 2^{-l}}\int_{\R^{n}}A\Big(c 2^{ls}|u(x+t e_{j})-u(x)|\Big)\dx\leq \sum_{i=1}^\infty\int_{\R^{n}} A\big(c_12^{is}\big|u_i\big|\big)\dx,
\end{aligned}
\end{align}
where $u_i$ is defined as in \eqref{march318}. To prove inequality \eqref{eq:2903},
consider a function $\eta\in C_{0}^{\infty}(\R^{n})$ such that, on setting $\eta_{i}(x)=\eta(2^{-i}x)$ for $x \in \rn$, one  has that  $\eta_{i}(x)=1$ for $x\in \mathrm{sprt}(\varphi_{i})$ and for $i\geq 0$.
Since
\begin{equation}\label{jan76}
\check\eta_i * u_i = \mathfrak F^{-1}(\eta_i) * \mathfrak F^{-1}(\varphi_i \widehat u)= \mathfrak F^{-1}(\eta_i\varphi_i \widehat u)=\mathfrak F^{-1}(\varphi_i \widehat u)=u_i
\end{equation}
{\color{black} for $i \geq 0$,}
we have that
\begin{align*}
u=\sum_{i=0}^\infty u_i=\sum_{i=0}^\infty\check\eta_i*u_i.
\end{align*}
Thus, 
 \begin{align}\label{eq:2502}
\sum_{l=1}^\infty& \sup_{t\leq 2^{-l}}\int_{\R^{n}}A\Big(c 2^{ls}|u(x+t e_{j})-u(x)|\Big)\dx
\\ \nonumber
&\leq
\sum_{l=1}^\infty \sup_{t\leq 2^{-l}}\int_{\R^{n}}A\Big(c2^{ls}\sum_{i=0}^\infty|u_i(x+t e_{j})-u_i(x)|\Big)\dx\\ \nonumber
&\leq
 \int_{\R^{n}}\sum_{l=1}^\infty A\Big(c\sum_{i=0}^l2^{{\color{black}(l-i)s}}\sup_{t\leq 2^{-l}}2^{is-l}\Big|\frac{\check\eta_i*u_i(x+t e_{j})-\check\eta_i*u_i(x)}{t}\Big|\Big)\dx\\ \nonumber
&\qquad+ \sum_{l=1}^\infty\sup_{t\leq 2^{-l}}\int_{\R^{n}} A\Big( c\sum_{i=l+1}^{\infty}2^{(l-i)s}2^{is}|u_i(x+t e_{j})-u_i(x)|\Big)\dx
\\ \nonumber
 &\leq
 \int_{\R^{n}}\sum_{l=1}^\infty A\Big(c\sum_{i=0}^l2^{(l-i)(s-1)}\sup_{t\leq 2^{-i}}2^{is-i}\Big|\frac{\check\eta_i*u_i(x+t e_{j})-\check\eta_i*u_i(x)}{t}\Big|\Big)\dx\\
 \nonumber &\qquad+ \int_{\R^{n}}\sum_{l=1}^\infty A\Big( c\sum_{i=l+1}^{\infty}2^{(l-i)s}2^{is}|u_i(x)|\Big)\dx\\
 \nonumber
 &\qquad+\sum_{l=1}^\infty\sup_{t\leq 2^{-l}}\int_{\R^{n}}\ A\Big( c\sum_{i=l+1}^{\infty}2^{(l-i)s}2^{is}|u_i(x+t e_{j})|\Big)\dx\\ \nonumber
 &=\mathrm{I}+\mathrm{II}+\mathrm{III}.
\end{align}
A change of coordinates  shows that
$\mathrm{II}=\mathrm{III}$. Therefore, it suffices to estimate the terms $\mathrm{I}$ and $\mathrm{II}$.  The inequality for discrete convolutions from Lemma \ref{lem:discrete} enters the game at this stage.
\\ First, consider $\mathrm{I}$.
Set $a_l^{\mathrm{I}}=2^{l(s-1)}$ for $l\geq0$ and $a_l^{\mathrm{I}}=0$ for $l<0$, and 
$$b_i^{\mathrm{I}}=\sup_{t\leq 2^{-i}}2^{is-i}\Big|\frac{\check\eta_i*u_i(x+t e_{j})-\check\eta_i*u_i(x)}{t}\Big| \quad \text{for $i\geq 0$,}$$
and $b_i^{\mathrm{I}}=0$ for $i<0$.
Since   $\{a_l^{\mathrm{I}}\}\in \ell^1(\mathbb Z)$, 
an application  of inequality \eqref{dec5}  tells us that
\begin{align}\label{march329}
\mathrm{I}&=\sum_{l\in\mathbb Z}A\bigg(c\sum_{i\in\mathbb Z}{\color{black} a_{l-i}^{\mathrm{I}}b_{i}^{\mathrm{I}}}\bigg)\\ \nonumber
 &\leq\int_{\R^{n}}\sum_{i=1}^\infty A\Big(c_1\sup_{t\leq 2^{-i}}2^{is-i}\Big|\frac{\check\eta_i*u_i(x+t e_{j})-\check\eta_i*u_i(x)}{t}\Big|\Big)\dx\\ \nonumber
 &\leq\sum_{i=1}^\infty\int_{\R^{n}} A\Big(c_1\Big|2^{-i}\sup_{t\leq 2^{-i}}t^{-1}|\check\eta_i(\cdot+te_j)-\check\eta_i|\Big|*|2^{is}u_i|\Big)\dx.
\end{align}
 Observe that
 \begin{align}\label{march330}
\Big\|\sup_{t\leq 2^{-i}}2^{-i}t^{-1}(\check\eta_i(\cdot+te_j)-\check\eta_i)\Big\|_{L^1(\R^n)}&\leq \bigg\|\sup_{t\leq 2^{-i}}\frac{1}{t}\int_0^t2^{-i}|\nabla \check\eta_i(\cdot+se_j)|\,\dd s\bigg\|_{L^1(\R^n)}\\ \nonumber
&= \bigg\|\sup_{t\leq 2^{-i}}\frac{1}{t}\int_0^t|\nabla \check\eta(\cdot+s2^ie_j)|\,\dd s\bigg\|_{L^1(\R^n)}\\ \nonumber
&\leq  \bigg\|\sup_{z\in B_1(0)}|\nabla \check\eta(\cdot+z)|\bigg\|_{L^1(\R^n)}\\ \nonumber
&=  \bigg\|\sup_{z\in B_1(0)}|\nabla \check\eta(\cdot+z)|\bigg\|_{L^1(B_2(0))}\\ \nonumber
& \quad + \bigg\|\sup_{z\in B_1(0)}\big(|\nabla \check\eta(\cdot+z)||\cdot+z|^{n+1}\big)|\cdot+z|^{-n-1}\bigg\|_{L^1(\R^n\setminus B_2(0))}\\   \nonumber &\leq c  +c \big\||\cdot|^{-n-1}\big\|_{L^1(\R^n\setminus B_1(0))}\leq\,c.
\end{align}
In equation \eqref{march330},  the equalities $\check\eta_i(\cdot)=2^{in}\check\eta(2^i\cdot)$ and  $\nabla \check\eta_i (\cdot)=2^{i(n+1)}\nabla \check\eta(2^i\cdot)$,  as well as  the membership $\check\eta\in\mathscr S(\R^n)$ play a role.
\\ Thanks to \eqref{march330}, 
the convolution inequality   \eqref{conv5} yields
\begin{align}\label{march331}
\sum_{i=1}^\infty&\int_{\R^{n}} A\Big(c_1 2^{-i}\sup_{t\leq 2^{-i}}\Big|t^{-1}(\check\eta_i(\cdot+te_j)-\check\eta_i)\Big|*|2^{is}u_i|\Big)\dx
\leq\sum_{i=1}^\infty\int_{\R^{n}} A\big(c_2 2^{is}|u_i|\big)\dx.
\end{align}
 Next, let us focus on the term $\mathrm{II}$.
Define the sequence $\{a_i^{\mathrm{II}}\}$ as
$a_i^{\mathrm{II}}=2^{is}$ for {\color{black} $i <0$} and $a_i^{\mathrm{II}}=0$ for {\color{black} $i\geq0$}, and  the sequence $\{b_i^{\mathrm{II}}\}$  as
$b_i^{\mathrm{II}}=2^{is}|u_i(x)|$. Inasmuch as  $\{a_l^{\mathrm{II}}\}\in \ell^1(\mathbb Z)$, from inequality \eqref{dec5} again we infer
that
\begin{align}\label{march335}
\mathrm{II}&=\sum_{l\in\mathbb Z}A\bigg(c\sum_{i\in\mathbb Z}a_{l-i}^{\mathrm{II}}b_i^{\mathrm{II}}\bigg)
\leq\sum_{i=1}^\infty\int_{\R^{n}} A\big(c_12^{is}\big|u_i\big|\big)\dx.
\end{align}
Combining inequalities \eqref{eq:2502}--\eqref{march335}
establishes inequality \eqref{eq:2903}. Inequality  \eqref{eq:2903A'}, for $s\in (0,1)$, follows from inequalities \eqref{march315} and \eqref{eq:2903}, via the first inequality in \eqref{eq:0603}.
\\ Assume  now that $s> 1$. A classical Poincar\'e inequality tells us that there exists a constant $c$, independent of $x$ and $r$, such that
\begin{equation}\label{poincare}
\inf_{\mathfrak q\in\mathcal P_{[ s]-1}}\dashint_{B_r(x)}\frac{|u-\mathfrak q|}{r^s}\dy \leq c\dashint_{B_r(x)}\frac{|\nabla ^{[s]}u|}{r^{\{s\}}}\dy
\end{equation}
for  $u\in W^{[s],1}(B_r(x))$.
On the other hand, for each function $u$ from this space there exists a polynomial $\mathfrak p \in \mathcal P_{[s]}$ such that 
\begin{equation}\label{poincare1}
\osc^{\{s\}} (\nabla^{[ s]} u)(x,r)= \dashint_{B_r(x)}\frac{|\nabla ^{[s]}(u -\mathfrak p ) |}{r^{\{s\}}}\dy.
\end{equation}
An application of inequality \eqref{poincare} with $u$ replaced by $u -\mathfrak p$ then yields
\begin{align}\label{poincare2}
\osc^s (u)(x,r)&\leq  \inf_{\mathfrak q\in\mathcal P_{[ s]-1}}\dashint_{B_r(x)}\frac{|u-\mathfrak p-\mathfrak q|}{r^s}\dy \leq c  \dashint_{B_r(x)}\frac{|\nabla ^{[s]}(u -\mathfrak p ) |}{r^{\{s\}}}\dy = c \osc^{\{s\}} (\nabla^{[ s]} u)(x,r).
\end{align}
Hence,
\begin{align*}
\osc^s (u)(x,r)\leq\,c\osc^{\{s\}} (\nabla^{[ s]} u)(x,r)
\end{align*}
for every  $x \in \rn$, $r>0$ and $u\in W^{[s],1}(B_r(x))$. Thus,
\begin{equation}\label{jan71}
\int_0^1\int_{\R^n}A\big(\osc^s u(x,r)\big)\dx\frac{\dd r}{r} \leq \int_0^1\int_{\R^n}A\big(c\osc^{\{s\}} (\nabla^{[ s]} u)(x,r))\big)\dx\frac{\dd r}{r}.
\end{equation}
One can then start from equation \eqref{jan50}, with 
 $s$ replaced by ${\color{black}\{s\}}$ and $u$ replaced by $\nabla^{[ s]}u$, and repeat the same steps to obtain
\begin{align*}
\int_0^1\int_{\R^n}A\big(\osc^s (u)(x,r)\big)\dx\frac{\dd r}{r}&\leq\sum_{i=0}^\infty\int_{\R^n} A(c 2^{i\{s\}}|\varphi_i(D)\nabla^{[ s]}u|)\dx\\
&\leq\sum_{i=0}^\infty\int_{\R^n} A(c_12^{is}|\varphi_i(D)u|)\dx.
\end{align*}
 Note that the last step also makes use of the
 second inequality in \eqref{eq:0603}. Inequality \eqref{eq:2903A'} is established for every $s>0$.
\\ Now, consider   inequality \eqref{eq:2903A''}. Fix any  $N>{[ s]}/2$  and  any function $\mathfrak K_0$ as in Lemma \ref{lem:support}  , and set 
$\mathfrak K=\Delta^N\mathfrak K_0$.  Hence, $\int_{\R^{n}}\mathfrak K(z)\mathfrak q(z)\,\dd z=0$.
Given any $i\geq1$ and any polynomial $\mathfrak q\in \mathcal P_{[ s]}$,  one has that
\begin{align*}
2^{in}\mathfrak K(2^{i}\cdot)*u (x)=\int_{\R^{n}}\mathfrak K(z)\big(u(x-2^{-i}z)-\mathfrak q(x-2^{-i}z)\big)\dd z \quad \text{for $x\in \rn$.}
\end{align*}
Furthermore,  if $i_0$ is so large that  $\mathrm{spt}(\mathfrak K_0)\subset B_{2^{i_0}}(0)$, then
\begin{align*}
|2^{in}\mathfrak K(2^{i}\cdot)*u(x)|&\leq \int_{B_{2^{i_0}}(0)}|\mathfrak K(z)\big|u(x-2^{-i}z)-\mathfrak q(x-2^{-i}z)\big|\dd z\\
&{\color{black} \leq }\,c \dashint_{B_{2^{i_0-i}}(0)}\big|u(x-z)-\mathfrak q(x-z)\big|\,\dd z
=\,c \dashint_{B_{2^{i_0-i}}(x)}\big|u(z)-\mathfrak q(z)\big|\,\dd z \quad \text{for $x\in \rn$.}
\end{align*}
Consequently, thanks to inequality \eqref{march313} and the arbitrariness of $\mathfrak q$,  
\begin{align*}
\int_{\R^n}A(2^{is}|\varphi_{i}(D)u|)\dx & \leq 
\int_{\R^{n}} A\bigg(c\, 2^{is}\inf_{\mathfrak q\in \mathcal P_{[ s]}}\dashint_{B_{2^{i_0-i}}(x)}\big|u(z)-\mathfrak q(z)\big|\,\dd z\bigg)\dx\\& = 
\int_{\R^{n}} A\big(c\osc^s u(x,2^{i_0-i})\big)\dx.
\end{align*}
Thus,
\begin{align}\label{jan70}
\sum_{i=1}^\infty\int_{\R^n}&A(2^{is}|\varphi_{i}(D)u|)\dx  \leq\sum_{i=1}^\infty \int_{\R^{n}} A\big(c\osc^s u(x,2^{i_0-i})\big)\dx\\ \nonumber
& \leq\sum_{i=1}^\infty \int_{\R^{n}} A\big(c_1\osc^s u(x,2^{-i})\big)\dx+\sum_{i=1}^{i_0} \int_{\R^{n}} A\big(c_1\osc^s u(x,2^{i})\big)\dx\\ \nonumber
&\leq \int_0^1\int_{\R^{n}} A\big(c_2\osc^s u(x,r)\big)\dx\frac{\dd r}{r}+
 \sum_{i=1}^{i_0} \int_{\R^n}\dashint _{B_{2^j}(x)}A\Big(c_1\frac{|u(y)|}{2^{js}}\Big)\dy\dx
\\ &\leq  \int_0^1\int_{\R^{n}} A\big(c_2\osc^s u(x,r)\big)\dx\frac{\dd r}{r}+
 \int_{\R^n}A(c_2|u(y)|)\dy.\nonumber
\end{align}
 where the last but one inequality holds by a discretization argument as in inequality \eqref{jan50} and by Jensen's inequality, and the last one by Fubini's theorem.
\\ Finally,   since $\check\varphi_{0}\in L^1(\R^n)$, inequality \eqref{eq:1903'} implies that
\begin{align}\label{jan74}
\int_{\R^n}A(|\varphi_{0}(D)u|)\dx=\int_{\R^n}A(|\check\varphi_{0}*u|)\dx\leq \int_{\R^n}A(c|u|)\dx.
\end{align}
Inequality \eqref{eq:2903A''} follows from \eqref{jan70} and \eqref{jan74}.
\end{proof}

\begin{remark}{\rm 
A close inspection of the proof shows that, in fact, we derived  the  inequality
\begin{align*}
\int_{\R^n}A(c|u|)\dx&+\int_0^1\int_{\R^n}\osc^s_A (cu)(x,r)\dx\frac{\dd r}{r}\leq\sum_{i=0}^\infty\int_{\R^n} A(2^{is}|\varphi_i(D)u|)\dx.
\end{align*}
 This inequality
 is stronger than  \eqref{eq:2903A'}.  On the other hand, from the latter, we deduce that
\begin{align*}
\int_{\R^n}A(c|u|)\dx&+\int_0^1\int_{\R^n}\osc^s_A (cu)(x,r)\dx\frac{\dd r}{r}\\
&\leq \int_{\R^n}A(|u|)\dx+\int_0^1\int_{\R^n}A\big(\osc^s (u)(x,r)\big)\dx\frac{\dd r}{r}
\end{align*}
  for every $u \in \mathcal M(\rn)$.
Since a reverse inequality holds by Jensen's inequality, it turns out that replacing  $A(\osc^s)$   by $\osc^s_A$ in its definition leaves the 
 space   $\mathcal O^{s,A}(\R^n)$ unchanged, up to equivalent norms.}
\end{remark}

\section{Equivalent norms defined via a Littlewood-Paley decomposition }\label{sec:LP}

 


 The equivalence of Gagliardo-Slobodecki norms and norms defined via a Littlewood-Paley decomposition is established here.  Theorem  \ref{besovequiv} serves as a bridge between these two definitions.

\begin{proof}[Proof of Theorem \ref{thm:equi}] 
{\color{black}
The conclusion will follow if we show that
\begin{align}\label{apr55}
\sum_{k=0}^{[ s]}\int_{\R^n} A(|\nabla^k u|)\dx+  \iint_{\R^{n}\times\R^{n}}A&\Big(\frac{|\nabla^{[ s]} u(x)-\nabla^{[ s]}u(y)|}{|x-y|^{\{s\}}}\Big)\,\frac{\dd\,(x,y)}{|x-y|^n}
\\  \nonumber & \leq \sum_{i=0}^{\infty}\int_{\R^n}A(c 2^{si}|\varphi_{i}(D)u|)\dx
\end{align}
for $u\in F^{s,A}(\R^n)$, and 
\begin{align}\label{apr56}
 \sum_{i=0}^{\infty}\int_{\R^n}&A(2^{si}|\varphi_{i}(D)u|)\dx
\\   \nonumber & \leq \sum_{k=0}^{[ s]}\int_{\R^n}A(c|\nabla^k u|)\dx+\iint_{\R^{n}\times\R^{n}}A\Big(c\frac{|\nabla^{[ s]}u(x)-\nabla^{[ s]}u(y)|}{|x-y|^{\{s\}}}\Big)\,\frac{\dd\,(x,y)}{|x-y|^n}
\end{align}
for $u \in W^{s,A}(\rn)$.
%
\\ Assume first that $s\in (0,1)$. 
 Lemma \ref{lem:BsA} ensures that, if $u\in F^{s,A}(\R^n)$, then $u \in L^A(\rn)$.  Plainly,
\begin{align}\label{april103}
\int_{0}^{\infty}\int_{\R^{n}}A\Big(\frac{|u(x+\varrho e_{j})-u(x)|}{\varrho^s}\Big)\dx\frac{\dd \varrho}{\varrho}&=\int_{0}^{1/2}\int_{\R^{n}}A\Big(\frac{|u(x+\varrho e_{j})-u(x)|}{\varrho^s}\Big)\dx\frac{\dd \varrho}{\varrho}\\ \nonumber
& \quad +\int_{1/2}^{\infty}\int_{\R^{n}}A\Big(\frac{|u(x+\varrho e_{j})-u(x)|}{\varrho^s}\Big)\dx\frac{\dd \varrho}{\varrho}.
\end{align}
By property \eqref{lambdaA}   and the    first inequality in \eqref{eq:0603},
\begin{align}\label{apr100}
\int_{1/2}^{\infty}\int_{\R^{n}}&A\Big(\frac{|u(x+\varrho e_{j})-u(x)|}{\varrho^s}\Big)\dx\frac{\dd \varrho}{\varrho} \leq \int_{1/2}^{\infty}\int_{\R^{n}}\Big(A(c|u(x+\varrho e_{j}|)+A(c|u(x)|)\Big)\dx\frac{\dd \varrho}{\varrho^{1+s}}\\ \nonumber
&\leq \int_{\R^{n}}A(c_1|u(x)|)\dx\int_{1/2}^{\infty}\frac{\dd \varrho}{\varrho^{1+s}}
\leq \int_{\R^{n}}A(c_2|u|)\dx \leq\sum_{i=0}^{\infty}\int_{\R^n}A(c_3 2^{is}|\varphi_{i}(D)u|)\dx.
\end{align}
Furthermore,   owing to inequality \eqref{eq:2903},
\begin{align}\label{eq:2903'}
\begin{aligned}
\int_{0}^{1/2}\int_{\R^{n}}&A\Big(\frac{|u(x+\varrho e_{j})-u(x)|}{\varrho^s}\Big)\dx\frac{\dd \varrho}{\varrho} 
=\sum_{l=1}^\infty\int_{\R^{n}}\int_{2^{-l-1}}^{2^{-l}}A\Big(\frac{|u(x+\varrho e_{j})-u(x)|}{\varrho^s}\Big)\frac{\dd \varrho}{\varrho}\dx\\
&\leq
\sum_{l=1}^\infty \sup_{t\leq 2^{-l}}\int_{\R^{n}}A\Big(c 2^{ls}|u(x+t e_{j})-u(x)|\Big)\dx \leq\sum_{i=0}^{\infty}\int_{\R^n}A(c_12^{is}|\varphi_{i}(D)u|)\dx.
\end{aligned}
\end{align}
As shown in the proof of Theorem \ref{besovequiv}, the left-hand side of inequality \eqref{april103}  is equivalent to
\begin{align}\label{apr150}
\int_{\R^{n}}\int_{\R^n}A\Big(c\frac{|u(x)-u(y)|}{|x-y|^s}\Big)\frac{\dd\,(x,y)}{|x-y|^n}.
\end{align}
Thereby, combining equations \eqref{april103}--\eqref{eq:2903'} and making use of  inequality   \eqref{eq:0603}  yield \eqref{apr55}.
\\
Next, fix $u \in B^{s,A}(\rn)$ and 
 let $i\geq 1$.
Thanks to Lemma \ref{lem:support}, it suffices to prove inequality \eqref{apr56} with $\int_{\R^n}A(2^{is}|\varphi_{i}(D)u|)\dx$ replaced by
$\int_{\R^n}A(2^{is}|2^{in}\mathfrak K(2^{i}\cdot)*u|)\dx$,
where $\mathfrak K$ is as in that lemma.
Since $\int_{\R^{n}}\mathfrak K(z)\dd z =0$,
we have that
\begin{align*}
2^{in}\mathfrak K(2^{i}\cdot)*u  (x)=\int_{\R^{n}}\mathfrak K(z)(u(x-2^{-i}z)-u(x))\dd z \quad \text{for $x \in \rn$.}
\end{align*}
Thus, Jensen's inequality entails that
\begin{align*}
\int_{\R^n}A(2^{is}|2^{in}\mathfrak K(2^{i}\cdot)*u|)\dx & \leq \int_{\R^{n}}|\mathfrak K(z)|\int_{\R^n}A(c2^{is}|u(x-2^{-i}z)-u(x)|)\dx\,\dd z \\ 
& \leq \int_{\R^n} |\mathfrak K(z)|\sup_{2^{-i-1}\leq t\leq 2^{-i}}\int_{\R^{n}}A\Big(c\frac{|u(x-tz)-u(x)|}{t^s}\Big)\dx\,\dd z.
\end{align*}
From inequality \eqref{march313} and the fact that   $\mathfrak K$ has compact support, one deduces that
\begin{align}\label{apr148}
\sum_{i=1}^{\infty}\int_{\R^n}A(2^{is}|\varphi_{i}(D)u|)\dx
&\leq  \int_{\R^n} |\mathfrak K(z)|\sum_{i=0}^\infty\sup_{2^{-i-1}\leq t\leq 2^{-i}}\int_{\R^{n}}A\Big(c\frac{|u(x-tz)-u(x)|}{t^s}\Big)\dx\,\dd z\\ \nonumber
&\leq \int_{\R^n} |\mathfrak K(z)|\int_{0}^{\infty}{\color{black}\sup_{r\leq t}\int_{\R^{n}}A\Big(c_1\frac{|u(x-rz)-u(x)|}{t^s}\Big)\dx}\frac{\dd t}{t}\,\dd z\\ \nonumber
&\leq\int_{\R^n} |\mathfrak K(z)|\int_{0}^{\infty}\sup_{r\leq t}\int_{\R^{n}}A\Big(c_2\frac{|u(x-rz/|z|)-u(x)|}{t^s}\Big)\dx\frac{\dd t}{t}\,\dd z\\ \nonumber
&\leq\int_{0}^{\infty}\sup_{|h|\leq t}\int_{\R^{n}}A\Big(c_3\frac{|u(x+h)-u(x)|}{t^s}\Big)\dx\frac{\dd t}{t}.
\end{align} 
The next step consists in removing the supremum in the last term. First of all, note that, for $h\in\R^n$ with $|h|< t$,
\begin{align*}
\int_{\R^{n}}A\Big(\frac{|u(x+h)-u(x)|}{t^s}\Big)\dx&\leq \int_{\R^{n}}A\Big(c\dashint_{B_t(0)}\frac{|u(x+y)-u(x)|}{t^s}\dy\Big)\dx\\& \quad +\int_{\R^{n}}A\Big(c\dashint_{B_t(0)}\frac{|u(x+y)-u(x+h)|}{t^s}\dy\Big)\dx\\
&= \int_{\R^{n}}A\Big(c\dashint_{B_t(0)}\frac{|u(x+y)-u(x)|}{t^s}\dy\Big)\dx\\&  \quad +\int_{\R^{n}}A\Big(c\dashint_{B_t(0)}\frac{|u(x)-u(x+y-h)|}{t^s}\dy\Big)\dx\\
&= \int_{\R^{n}}A\Big(c\dashint_{B_t(0)}\frac{|u(x+y)-u(x)|}{t^s}\dy\Big)\dx\\&  \quad+\int_{\R^{n}}A\Big(c\dashint_{B_t(-h)}\frac{|u(x)-u(x+y)|}{t^s}\dy\Big)\dx\\
&\leq \int_{\R^{n}}A\Big(c_1\dashint_{B_{2t}(0)}\frac{|u(x+y)-u(x)|}{t^s}\dy\Big)\dx.
\end{align*}
Hence,
\begin{align}\label{apr145}
\int_{0}^{\infty}\sup_{|h|\leq t}&\int_{\R^{n}}A\Big(\frac{|u(x+h)-u(x)|}{t^s}\Big)\dx\frac{\dd t}{t}\\ \nonumber
&\leq \int_{0}^{\infty}\int_{\R^{n}}A\Big(c_1\dashint_{B_{2t}(0)}\frac{|u(x+y)-u(x)|}{t^s}\dy\Big)\dx\frac{\dd t}{t}\\ \nonumber
&\leq \int_{0}^{\infty}\int_{\R^{n}}A\Big(\frac{c_2}{t}\int_{0}^{\color{black} 2t}\int_{\partial B_\varrho(0)}\frac{|u(x+y)-u(x)|}{t^s}\frac{{\color{black} {\rm d} \mathcal H^{n-1}(y)}}{|y|^{n-1}}\,\dd\varrho \Big)\dx\frac{\dd t}{t}\\ \nonumber
&= \int_{0}^{\infty}\int_{\R^{n}}A\Big(\frac{c_3}{t}\int_{0}^{\color{black} t}\int_{\partial B_1(0)}\frac{|u(x+\varrho y)-u(x)|}{t^s}{\color{black} {\rm d} \mathcal H^{n-1}(y)}\,\dd\varrho\Big)\dx\frac{\dd t}{t}.
\end{align}
Next, observe that the following Hardy-type inequality
\begin{align}\label{apr140}
\int_0^\infty A\Big(\frac{1}{t^{1+s}}\int_0^t f(\varrho) \,\dd\varrho\Big)\frac{\dt}{t}\leq \int_0^\infty A\Big(\frac{f(t)}{t^s}\Big)\frac{\dt}{t}
\end{align}
holds for every measurable
 function $f: [0, \infty) \to [0, \infty)$.  This is a consequence of the chain
\begin{align}\label{apr141}
\int_0^\infty A\Big(\frac{1}{t^{1+s}}\int_0^tf(\varrho)\,\dd\varrho\Big)\frac{\dt}{t} & \leq 
\int_0^\infty A\Big(\frac{1}{t}\int_0^t\frac{f(\varrho)}{\varrho^s}\,\dd\varrho\Big)\frac{\dt}{t} \leq
\int_0^\infty \frac{1}{t} \int_0^tA\Big(\frac{f(\varrho)}{\varrho^s}\Big)\,\dd\varrho\frac{\dt}{t} 
\\ \nonumber & = \int_0^\infty A\Big(\frac{f(\varrho)}{\varrho^s}\Big) \int_\varrho ^\infty \frac{\dt}{t^2}\, \dd\varrho = \int_0^\infty A\Big(\frac{f(t)}{t^s}\Big)\frac{\dt}{t},
\end{align}
where the second inequality holds by Jensen's inequality.
\\ From inequality \eqref{apr145}, via an application of inequality \eqref{apr141} with
\begin{align*}
f(\varrho)=\int_{\partial B_1(0)}|u(x+\varrho y)-u(x)|\dy \quad \text{for $\varrho \geq 0$,}
\end{align*}
and the use of Jensen's inequality again, one can deduce that
\begin{align}\label{apr147}
\int_{0}^{\infty}\sup_{|h|\leq t}&\int_{\R^{n}}A\Big(\frac{|u(x+h)-u(x)|}{t^s}\Big)\dx\frac{\dd t}{t}\\ \nonumber
&\leq  \int_{0}^{\infty}\int_{\R^{n}}A\Big(c\int_{\partial B_1(0)}\frac{|u(x+t y)-u(x)|}{t^s}{\color{black} {\rm d} \mathcal H^{n-1}(y)}\Big)\dx\frac{\dd t}{t}\\ \nonumber
&\leq  \int_{0}^{\infty}\int_{\R^{n}}\int_{\partial B_1(0)}A\Big(c_1\frac{|u(x+t y)-u(x)|}{t^s}\Big){\color{black} {\rm d} \mathcal H^{n-1}(y)}\dx\frac{\dd t}{t}\\ \nonumber
&= \int_{\R^{n}}\int_{0}^{\infty}\int_{\partial B_t(0)}A\Big(c_1\frac{|u(x+y)-u(x)|}{t^s}\Big)\frac{{\color{black} {\rm d} \mathcal H^{n-1}(y)}}{t^{n-1}}\frac{\dd t}{t}\dx\\ \nonumber
&= \int_{\R^{n}}\int_{\R^n}A\Big(c_1\frac{|u(x+y)-u(x)|}{|y|^s}\Big)\frac{\dy}{|y|^n}\dx
\\ \nonumber
&= \int_{\R^{n}}\int_{\R^n}A\Big(c_1\frac{|u(x)-u(y)|}{|x-y|^s}\Big)\frac{\dd\,(x,y)}{|x-y|^n}.
\end{align}
Inequalities \eqref{apr148}, \eqref{apr145} and \eqref{apr147} imply that
\begin{align}\label{apr130}
\sum_{i=1}^{\infty}\int_{\R^n}A(2^{is}|\varphi_{i}(D)u|)\dx \leq  \int_{\R^{n}}\int_{\R^n}A\Big(c\frac{|u(x)-u(y)|}{|x-y|^s}\Big)\frac{\dd\,(x,y)}{|x-y|^n}.
\end{align}
Combining the latter inequality with   \eqref{jan74} yields inequality  \eqref{apr56}. 
%
\\ Assume now that 
$s>1$. Let $u \in F^{s,A}(\rn)$. From inequality \eqref{eq:0603} applied with $u$ replaced by $\nabla ^k u$  and $s$  by $\{s\}$, 
and inequality \eqref{apr55}  with $u$ replaced by  $\nabla^{[ s]}u$ and $s$  by $\{s\}$, one obtains that
%
%
\begin{align}\label{apr153}
\sum_{k=0}^{[ s]}&\int_{\R^n}A(|\nabla^k u|)\dx+\iint_{\R^{n}\times\R^{n}}A\Big(\frac{|\nabla^{[ s]} u(x)-\nabla^{[ s]}u(y)|}{|x-y|^{\{s\}}}\Big)\,\frac{\dd\,(x,y)}{|x-y|^n}
\\ \nonumber
&\leq\sum_{k=0}^{[ s]}\sum_{i=0}^\infty\int_{\R^n}A\big(c2^{i\{s\}}\big|\varphi_i(D)\nabla^k u\big|\big)\dx+\sum_{i=0}^\infty\int_{\R^n}A\big(c2^{i\{s\}}\big|\varphi_i(D)\nabla^{[ s]}u\big|\big)\dx.
\end{align}
Moreover, by the second inequality in   \eqref{eq:0603} and the  monotonicity of the function $A$, 
\begin{align}\label{apr154}
\sum_{k=0}^{[ s]}\sum_{i=0}^\infty&\int_{\R^n}A\big(c2^{i\{s\}}\big|\varphi_i(D)\nabla^k u\big|\big)\dx+\sum_{i=0}^\infty\int_{\R^n}A\big(c2^{i\{s\}}\big|\varphi_i(D)\nabla^{[ s]}u\big|\big)\dx \\ \nonumber
&\leq\sum_{k=0}^{[ s]}\sum_{i=0}^\infty\int_{\R^n}A\big(c2^{i(k+\{s\})}\big|\varphi_i(D)u\big|\big)\dx+\sum_{i=0}^\infty\int_{\R^n}A(c2^{is}|\varphi_i(D)u|)\dx\\ \nonumber
&\leq\sum_{i=0}^\infty\int_{\R^n}A(c_12^{is}|\varphi_i(D)u|)\dx. 
\end{align} 
 Inequality  \eqref{apr55} follows from \eqref{apr153} and \eqref{apr154}. 
\\ As far as   inequality \eqref{apr56} is concerned,}
owing to inequalities \eqref{jan70} and \eqref{jan71}, 
we have that
\begin{align}\label{apr155}
\sum_{i=1}^\infty\int_{\R^n}A(2^{si}|\varphi_{i}(D)u|)\dx 
&\leq \int_0^1\int_{\R^{n}} A\big(c\osc^s u(x,r)\big)\dx\frac{\dd r}{r}\\ \nonumber
&\leq \int_0^1\int_{\R^{n}} A\big(c_1\osc^{\{s\}} \nabla^{[s]} u(x,r)\big)\dx\frac{\dd r}{r}.
\end{align}
 Clearly,
\begin{align}\label{apr156}
\osc^{\{s\}} \nabla^{[s]} u(x,r)&{\color{black} \leq }\dashint_{B_r(x)}\frac{|\nabla^{[s]}u(y)-(\nabla^{[s]}u)_{B_r(x)}|}{r^{\{s\}}}\dy \\ \nonumber
& \leq\dashint_{B_r(x)}\dashint_{B_r(x)}\frac{|\nabla^{[s]}u(y)-\nabla^{[s]}u(z)|}{r^{\{s\}}}\dy\,\dd z\\ \nonumber
&\leq\dashint_{B_r(x)}\frac{|\nabla^{[s]}u(y)-\nabla^{[s]}u(x)|}{r^{\{s\}}}\dy+\dashint_{B_r(x)}\frac{|\nabla^{[s]}u(y)-\nabla^{[s]}u(x)|}{r^{\{s\}}}\,\dd z\\ \nonumber
&=2\dashint_{B_r(0)}\frac{|\nabla^{[s]}u(x+y)-\nabla^{[s]}u(x)|}{r^{\{s\}}}\dy.
\end{align}
Exploiting  inequalities \eqref{apr155} and \eqref{apr156} and  Jensen's inequality enable one to infer that
\begin{align}\label{jan72}
\sum_{i=1}^\infty\int_{\R^n}A(2^{si}|\varphi_{i}(D)u|)\dx &\leq\int_0^\infty\int_{\R^n} \dashint_{B_r(0)}A\bigg(c\frac{|\nabla^{[s]}u(x+y)-\nabla^{[s]}u(x)|}{r^{\{s\}}}\bigg)\dy\dx\frac{\dd r}{r}\\
&=\int_0^\infty\dashint_{B_r(0)}\int_{\R^n} A\bigg(c\frac{|\nabla^{[s]}u(x+y)-\nabla^{[s]}u(x)|}{r^{\{s\}}}\bigg)\dx\dy\frac{\dd r}{r}\nonumber\\ 
&\leq\int_0^\infty\sup_{|y|<t}\int_{\R^n} A\bigg(c\frac{|\nabla^{[s]}u(x+y)-\nabla^{[s]}u(x)|}{r^s}\bigg)\dx\frac{\dd r}{r}. \nonumber
\end{align}
An application of inequality \eqref{apr147} with $u$ replaced by $\nabla^{[s]}u$ and $s$ by $\{s\}$ yields
  \begin{align}\label{jan73}
\int_0^\infty \sup_{|y|<t}\int_{\R^n} & A\bigg(c\frac{|\nabla^{[s]}u(x+y)-\nabla^{[s]}u(x)|}{r^s}\bigg)\dx\frac{\dd r}{r} \\ \nonumber &  \leq 
 \int_{\R^{n}}\int_{\R^n}A\Big(c_1\frac{|\nabla^{[s]}u(x+y)-\nabla^{[s]}u(x)|}{|y|^{\{s\}}}\Big)\frac{\dy}{|y|^n}\dx
\\ \nonumber
&=  
\int_{\R^{n}}\int_{\R^n}A\Big(c_1\frac{|\nabla^{[s]}u(x)-\nabla^{[s]}u(y)|}{|x-y|^{\{s\}}}\Big)\frac{\dy}{|x-y|^n}\dx.
\end{align}
{\color{black} Inequality in \eqref{apr56} follows from \eqref{jan72}, \eqref{jan73} and \eqref{jan74}.}
\end{proof}

\section{Proof of Theorem \ref{thm:main}}\label{sec:emb}


We are now in a position to prove our main result.

\begin{proof}[Proof of Theorem \ref{thm:main}] Owing to Theorem \ref{thm:equi}, it suffices to prove  inequality \eqref{normineq} with the norms in $W^{s,A}(\rn)$ and  $W^{\frac n{s-r},A}(\rn)$ replaced by those in $F^{s,A}(\rn)$ and  $F^{\frac n{s-r},A}(\rn)$.
\\
Let  $\eta\in C_{0}^{\infty}(\R^{n})$ be a function as in the proof of Theorem \ref{thmosc}, namely such that $\eta_{i}(x)=\eta(2^{-i}x)$ equals $1$ in $\mathrm{spt}(\varphi_{i})$ for $i\geq 1$. Assume that  
$u\in F^{A,s}(\R^{n})$.
Thanks to equation \eqref{jan76},
\begin{align*}
\int_{\R^n}A_{\frac{n}{s-r}}(2^{ir}|\varphi_i(D)u|)\dx  = \int_{\R^n}A_{\frac{n}{s-r}}(|(2^{i(r-s)}\check{\eta}_{i})*(2^{is}\varphi_{i}(D)u)|)\dx
\end{align*}
 for $i\geq 1$.
On setting $p=\frac n{n-s+r}$, one obtains  that
\begin{align*}
\|\check{\eta}_{i}\|_{L^{p}(\R^{n})} & = 2^{in}\Big(\int_{\R^{n}}|\check{\eta}(2^{i}z)|^{p}\,\mathrm{d} z \Big)^{\frac{1}{p}} = c2^{\frac{in}{p'}}. 
\end{align*}
Inasmuch as
$\frac{in}{p'}=i(s-r)$, we have that $\|(2^{i(r-s)}\check{\eta}_{i}\|_{L^p(\R^n)}\leq c$ for $i\geq 1$.. An application of   the convolution inequality  \eqref{conv5} (with $s$ replaced by $s-r$) yields
\begin{align}\label{eq:convolution}
 \int_{\R^n}A_{\frac{n}{s-r}}(|(2^{i(r-s)}\check{\eta}_{i})*(2^{is}\varphi_{i}(D)u)|)\dx\leq  \int_{\R^n}A(c|2^{is}\varphi_{i}(D)u|)\dx.
\end{align}
On the other hand, {\color{black} replacing $\eta_i$ by a compactly supported function which equals $1$ in the support of $\varphi_0$} and arguing as above tell us that 
\begin{align}\label{apr160}
\int_{\R^n}A_{\frac{n}{s-r}}(|\varphi_0(D)u|)\dx\leq \int_{\R^n}A(c|\varphi_{0}(D)u|)\dx.
\end{align}
Form inequalities \eqref{eq:convolution} and \eqref{apr160} one infers that
\begin{align*}
 \sum_{i=0}^\infty\int_{\R^n}A_{\frac{n}{s-r}}(2^{ir}|\varphi_i(D)u|)\dx\leq  \sum_{i=0}^\infty\int_{\R^n}A(c|2^{is}\varphi_{i}(D)u|)\dx.
\end{align*}
Hence, inequality \eqref{eq:thmmain} follows, via the definition 
of the norm in $F^{s,A}(\rn)$.
\end{proof}

\section*{Compliance with Ethical Standards}\label{conflicts}
\subsection*{Funding}

This research was partly funded by:

\begin{enumerate}
\item Research Project 201758MTR2  of the Italian Ministry of Education, University and
Research (MIUR) Prin 2017 ``Direct and inverse problems for partial differential equations: theoretical aspects and applications'';
\item GNAMPA of the Italian INdAM -- National Institute of High Mathematics
(grant number not available).
\end{enumerate}

\subsection*{Conflict of Interest}

The authors declare that they have no conflict of interest.

\end{document}